\documentclass[12pt]{article}
\usepackage{amsmath}
\usepackage{amssymb}
\usepackage{latexsym}
\usepackage{graphicx}

\title
{Bernstein measures on convex polytopes}
\author
{
Tatsuya Tate\thanks{Research partially supported by 
JSPS Grant-in-Aid for Scientific Research (No. 18740089).}\\
Graduate School of Mathematics \\
Nagoya University \\
Furo-cho, Chikusa-ku, \\
Nagoya, 464--8602, Japan\\
Email: tate@math.nagoya-u.ac.jp}
\date{\empty}

\renewcommand{\phi}{\varphi}

\newcommand{\ispa}[1]{\langle \,#1 \,\rangle }

\newtheorem{theorem}{{\sc Theorem}}[section]

\newtheorem{lemma}[theorem]{{\sc Lemma}}
\newtheorem{proposition}[theorem]{{\sc Proposition}}

\newtheorem{definition}[theorem]{{\sc Definition}}
\newtheorem{example}[theorem]{{\sc Example}}
\newtheorem{remark}[theorem]{{\sc Remark}}

\newenvironment{proof}%
{\def\psymbol{{\sc Proof.}\enspace}
\psymbol}%
{\def\qed{$\square$}
\hspace{3pt}\qed\par\bigskip}
\begin{document}

\maketitle

\renewcommand{\labelenumi}{{\rm (\arabic{enumi})}}

%\curraddr{Department of Mathematics and Statistics,
%Case Western Reserve University, Cleveland, Ohio 43403}
%\email{tate@math.nagoya-u.ac.jp}
%\thanks{Research partially supported by JSPS Grant-in-Aid for Scientific Research (No.~18740089).}

%\subjclass{Primary 41A60, 52B20; Secondary 53D05, 26B25}
%\date{\today}

%\dedicatory{Dedicated to professor Toshikazu Sunada on the occasion of his sixtieth birthday}

%\keywords{Bernstein polynomial, Bernstein measure, asymptotic expansion, polytope, toric variety}

\begin{abstract}
We define the notion of Bernstein measures and Bernstein approximations over 
general convex polytopes. This generalizes well-known Bernstein polynomials 
which are used to prove the Weierstrass approximation theorem on one dimensional intervals. 
We discuss some properties of Bernstein measures and approximations, 
and prove an asymptotic expansion of the Bernstein approximations for smooth functions 
which is a generalization of the asymptotic expansion of the Bernstein polynomials 
on the standard $m$-simplex obtained by Abel-Ivan and H\"{o}rmander. These are different from the 
Bergman-Bernstein approximations over Delzant polytopes recently introduced by Zelditch. 
We discuss relations between Bernstein approximations defined in this paper and 
Zelditch's Bergman-Bernstein approximations. 
\end{abstract}

\section{Introduction}
\label{intro}
It is quite well-known that Bernstein (\cite{B}) introduced the polynomial 
\begin{equation}
\label{original}
B_{N}(f)(x)=\sum_{k=0}^{N}{N \choose k}x^{k}(1-x)^{N-k}f(k/N),\quad 
f:[0,1] \to \mathbb{C}
\end{equation}
to prove the Weierstrass approximation theorem. 
After his work, several properties and generalizations of the Bernstein polynomials \eqref{original} are
considered. For example, the Bernstein polynomials 
on the standard $m$-dimensional simplex $P$ is defined by 
\begin{equation}
\label{simplex}
B_{N}(f)(x)=\sum_{\alpha \in \mathbb{Z}_{ \geq 0}^{m},\ 
\|\alpha\| \leq N}{N \choose \alpha}x^{\alpha}(1-\|x\|)^{N-\|\alpha\|}f(\alpha/N), 
\end{equation}
where $f$ is a function on $P$, 
$\alpha= (\alpha_{1},\ldots,\alpha_{m})\in \mathbb{Z}_{\geq 0}^{m}$ and
$\|\alpha\|=\sum_{j=1}^{m}\alpha_{j}$.  
The reader may be referred to the book \cite{Lo} in which one can find many interesting 
properties and some generalizations of the Bernstein polynomials. 
Some relations between Bernstein polynomials and Brownian motions are discussed in \cite{K}. 

Rather recently, Abel-Ivan (\cite{AI}) and 
H\"{o}rmander (\cite{Ho2}) obtained a complete asymptotic expansion of 
the Bernstein polynomials on the standard $m$-dimensional simplex defined by \eqref{simplex}. 
In particular, H\"{o}rmander used it to analyze asymptotic behavior of certain 
Bergman kernels near the boundary of the unit ball in a complex space. 

Further generalization was considered by Zelditch (\cite{Z}, see also \cite{Fe}), who defined the Bergman-Bernstein 
approximations on regular simple lattice ({\it i.e.}, Delzant) polytopes. 
He defined it by using Bergman-Szeg\"{o} kernels on a toric K\"{a}hler manifold 
whose moment polytope is a given Delzant polytope. He obtained an asymptotic 
expansion of the Bergman-Bernstein approximations by using the theory of 
Toeplitz operators and the harmonic analysis on toric K\"{a}hler manifolds developed in \cite{STZ1}, \cite{STZ2}, 
\cite{SoZ1}, and applied the Bergman-Bernstein approximations to 
obtain an asymptotic expansion of the Dedekind-Riemann sum 
\begin{equation}
\label{Riemann}
\sum_{\alpha \in NP\cap \mathbb{Z}^{m}}f(\alpha/N)
\end{equation}
over the Delzant polytope $P$ in $\mathbb{R}^{m}$. 
Note that an asymptotic expansion of the Dedekind-Riemann sum \eqref{Riemann}, 
for simple lattice polytopes, in a form similar to formulas of Euler-Maclaurin type 
was obtained by Guillemin-Sternberg \cite{GS}. 

In view of this, it would be interesting to consider some generalizations of the Bernstein 
polynomials for general convex polytopes and its applications to various directions. 
Our aim in this paper is to define in a natural way a generalization 
of the Bernstein polynomials on general convex polytopes and obtain its asymptotic expansion. 
We avoid to use harmonic analysis on toric varieties to define the Bernstein approximations 
on general polytopes, because the toric variety corresponding to a general polytope will have singularities. 
Instead, we use the setting-up and some analysis discussed in \cite{TZ}. 
Our main theorems are described in Section \ref{Bernstein} and proved in Section \ref{property}. 

Our Bernstein approximations defined in Section \ref{Bernstein} are different from Bergman-Bernstein 
approximations defined in \cite{Z}. In \cite{Z}, Zelditch used a general K\"{a}hler metric on a toric K\"{a}hler manifold. 
One may say that, when the polytope is Delzant, 
we use only the K\"{a}hler metric induced by the restriction of the Fubini-Study 
metric through a monomial embedding. 
But, our Bernstein approximations can be defined for general polytopes. 
Recently, Song-Zelditch (\cite{SoZ2}) shows a large deviations principle for 
the Bergman-Bernstein measures. In particular, they give a concrete description of the rate functions.
Our Bernstein approximations are defined in terms of dilated convolution powers of 
Bernstein measures, and hence they satisfy a large deviations principle. 
We give a concrete description of the rate functions for finitely supported Bernstein measures 
in Subsection \ref{LargeDP}, 
which eventually coincide with Song-Zelditch rate functions (\cite{SoZ2}) for Bergman-Bernstein measures 
defined from Fubini-Study metric on a projective toric manifold when the polytope is Delzant. 
We discuss in Section \ref{ToricGeometry} relations between our Bernstein approximations 
and Bergman-Bernstein approximations.

\vspace{20pt}

{\bf Acknowledgments.} The author would like to thank to professor Zelditch for his 
valuable comments on the topics discussed here, 
and professor Bando for his pointing out to the author on the facts given in Subsection \ref{generalS}.

\section{Bernstein measures on convex polytopes}
\label{Bernstein}

\subsection{Definitions and main theorems}
\label{mains}

Let $P \subset \mathbb{R}^{m}$ be a polytope which has non-empty interior, $P^{o}$. 
The barycenter,
\begin{equation}
\label{barycenter}
b(\mu)=\int_{\mathbb{R}^{m}}z\,d\mu(z), 
\end{equation}
of a given probability measure $\mu$ on $\mathbb{R}^{m}$ 
is contained in the convex hull of the support, ${\rm supp}(\mu)$, of $\mu$, 
and hence it defines a continuous surjection 
\begin{equation}
\label{barymap}
b:\mathcal{M}(P) \to P, 
\end{equation}
where $\mathcal{M}(P)$ denotes the set of probability measures on the polytope $P$ with 
the weak-$\ast$ topology. We denote $C(P)$ the space of all continuous functions on $P$. 

\begin{definition}
\label{Bmeasure}
A continuous section $\mathcal{B}:P \to \mathcal{M}(P)$ of the barycenter map \eqref{barymap} 
satisfying the following two conditions is called the {\it Bernstein measure} on the polytope $P$: 
\begin{enumerate}
\item For any $f \in C(P)$, the function $B(f)$, defined by 
\begin{equation}
\label{integralB}
B(f)(x):=\int_{P}f(z)\,d\mathcal{B}_{x}(z),\quad d\mathcal{B}_{x}:=\mathcal{B}(x) \in \mathcal{M}(P), 
\quad x \in P, 
\end{equation}
is in $C^{\infty}(P^{o}) \cap C(P)$.

\item There exists a smooth function $K:P^{o} \to {\rm Sym}(m,\mathbb{R})$ such that 
for any $f \in C(P)$, we have 
\[
\nabla B(f)(x)=\int_{P} f(z) K(x)(z-x)\,d\mathcal{B}_{x}(z), 
\quad x \in P^{o},  
\]
where ${\rm Sym}(m,\mathbb{R})$ denotes the space of all symmetric $m \times m$ real matrices. 
We call the function $K :P^{o} \to {\rm Sym}(m,\mathbb{R})$ the {\it defining matrix} of the Bernstein measure $\mathcal{B}$. 
\end{enumerate}
\end{definition}

\begin{definition}
\label{Bapprox}
Let $\mathcal{B}:P \to \mathcal{M}(P)$ be a Bernstein measure on $P$. 
For any $f \in C(P)$ and any positive integer $N$, we define the function $B_{N}(f)$ on $P$ by 
\begin{equation}
\label{BNf}
B_{N}(f)(x):=\int_{P}f(z)\,d\mathcal{B}_{x}^{N}(z), 
\end{equation}
where $d\mathcal{B}_{x}^{N}$ is the probability measure on $P$ given by 
\begin{equation}
\label{convolution}
d\mathcal{B}_{x}^{N}=(D_{1/N})_{*}(d\mathcal{B}_{x} \ast \cdots \ast d\mathcal{B}_{x}), 
\quad D_{1/N}:\mathbb{R}^{m} \ni x \mapsto x/N \in \mathbb{R}^{m}. 
\end{equation}
We call the function $B_{N}(f)$ the {\it Bernstein approximation} of $f$. 
\end{definition}

Concrete expression for the dilated convolution measure $d\mathcal{B}_{x}^{N}$ of $\mathcal{B}(x)$ 
defined in \eqref{convolution} is given by 
\begin{equation}
\label{remark4}
\int_{P}f(z)\,d\mathcal{B}_{x}^{N}(z)=
\int_{P \times \cdots \times P}
f((z^{(1)}+\cdots+ z^{(N)})/N)\,d\mathcal{B}_{x}(z^{(1)})\cdots d\mathcal{B}_{x}(z^{(N)}),  
\end{equation}
where the integral domain $P \times \cdots \times P$ is the $N$ times product of the polytope $P$.

To our knowledge, generalizations of the classical Bernstein polynomials 
previously considered are defined by using finitely supported probability measures. This is true also for 
the Bergman-Bernstein approximation in \cite{Z}. We do not assume here that 
the measure $d\mathcal{B}_{x}$ to have finite support. 
However, a Bernstein measure 
having a finite support in the following sense might be important. 

\begin{definition}
\label{defSupport}
Let $\mathcal{B}:P \to \mathcal{M}(P)$ be a Bernstein measure. 
Then, $\mathcal{B}$ is called a {\it finitely supported} Bernstein measure if there exists a finite set $S$ in $P$ such that 
${\rm supp}(B(x))=S$ for any $x \in P^{o}$ and ${\rm supp}(B(x))\subset S$ for any $x \in P$. 
The finite set $S$ is called the {\it support} of $\mathcal{B}$. 
\end{definition}

Now, let $\mathcal{B}:P \to \mathcal{M}(P)$ be a finitely supported Bernstein measure with support $S \subset P$. 
Then, it is easy to show that, for each $\alpha \in S$, 
there is a function $m_{\alpha} \in C^{\infty}(P^{o}) \cap C(P)$ such that 
$d\mathcal{B}_{x}=\sum_{\alpha \in S}m_{\alpha}(x)\delta_{\alpha}$ for each $x \in P$. 
The functions $m_{\alpha}$, $\alpha \in S$ satisfy the following: 
\begin{enumerate}
\item $m_{\alpha}$ is positive on $P^{o}$, and we have $\sum_{\alpha \in S}m_{\alpha}(x)=1$ for all $x \in P$; 
\item for each $x \in P$, we have $\sum_{\alpha \in S}m_{\alpha}(x)\alpha=x$, which 
implies that the convex hull of $S$ is $P$;  
\item we have $\nabla m_{\alpha}(x)=m_{\alpha}(x)K(x)(\alpha -x)$ for each $x \in P^{o}$. 
\end{enumerate}
Conversely, a finite set $S \subset P$ whose convex hull is $P$ and 
functions $m_{\alpha} \in  C^{\infty}(P^{o}) \cap C(P)$ for each $\alpha \in S$ satisfying the above three 
conditions define a Bernstein measure $\mathcal{B}(x)=\sum_{\alpha \in S}m_{\alpha}(x)\delta_{\alpha}$.

The following lemma gives a concrete expression for the dilated convolution powers $d\mathcal{B}_{x}^{N}$ 
for a given finitely supported Bernstein measure $d\mathcal{B}_{x}$.

\begin{lemma}
\label{concFBM}
Let $\mathcal{B}:P \to \mathcal{M}(P)$ be a finitely supported Bernstein measure with the support $S$, 
and write it as $\mathcal{B}(x)=\sum_{\alpha \in S}m_{\alpha}(x)\delta_{\alpha}$. 
Then, the probability measure $d\mathcal{B}_{x}^{N}$ is written as 
\begin{equation}
\label{DCM}
d\mathcal{B}_{x}^{N}=\sum_{\gamma \in S_{N}}m_{N}^{\gamma}(x)\delta_{\gamma/N},
\end{equation}
where $S_{N} \subset NP$ is the finite set defined by 
\begin{equation}
\label{achieval}
S_{N}=\{\gamma \in NP\,;\,\gamma =\beta_{1}+\cdots +\beta_{N}\ 
\mbox{for some}\ \beta_{1},\ldots,\beta_{N} \in S\}, 
\end{equation}
and the function $m_{N}^{\gamma} \in C(P) \cap C^{\infty}(P^{o})$ is defined by 
\begin{equation}
\label{coeff}
m_{N}^{\gamma}(x)=\sum_{
\stackrel{\beta_{1},\ldots,\beta_{N} \in S}{
\beta_{1}+\cdots+\beta_{N} =\gamma}}
m_{\beta_{1}}(x) \cdots m_{\beta_{N}}(x). 
\end{equation}
\end{lemma}

\begin{proof}
By \eqref{remark4}, the Bernstein approximation $B_{N}(f)$ is given by 
\[
\begin{split}
B_{N}(f)(x)&=\sum_{\beta_{1},\ldots,\beta_{N} \in S}
f((\beta_{1}+\cdots +\beta_{N})/N)m_{\beta_{1}}(x) \cdots m_{\beta_{N}}(x) \\
&=\sum_{\gamma \in S_{N}}
\sum_{
\stackrel{\beta_{1},\ldots, \beta_{N}}{\beta_{1}+\cdots+\beta_{N}=\gamma}}
f(\gamma/N)m_{\beta_{1}}(x) \cdots m_{\beta_{N}}(x),
\end{split}
\]
which shows the equations \eqref{DCM}, \eqref{coeff}. 
\end{proof}

We mention some remarks 
on the definition of Bernstein measures and Bernstein approximations. 

\begin{remark}
\label{remark1}
{\rm It is easy to show that the Bernstein approximation $B_{N}(f)$ for $f \in C(P)$ with 
respect to a Bernstein measure $\mathcal{B}:P \to \mathcal{M}(P)$ is in $C^{\infty}(P^{o}) \cap C(P)$. 
The condition (2) in Definition \ref{Bmeasure} comes from H\"{o}rmander's proof (\cite{Ho2}) of 
an asymptotic expansion of the Bernstein polynomials \eqref{simplex} on an $m$-dimensional simplex. 
In fact, he uses a differential-recurrence formula for a family of functions 
(which is, in our setting-up, the function $I_{N,\alpha}$ defined in \eqref{functI}). 
In his proof of the recurrence formula, he uses the matrix `$A(x)$' given in Example \ref{example1} below. 
The recurrence formula for the functions $I_{N,\alpha}$ defined in \eqref{functI} is given in Lemma \ref{recursion}. 
The condition (2) in Definition \ref{Bmeasure} assures that the functions $I_{N,\alpha}$ are 
polynomials in $N$ (Lemma \ref{polyn}), and it provides a computable representation for each differential operator in 
the asymptotic expansion \eqref{expansion} in Theorem \ref{AsympExp}. 

In general, for a given section $\mathcal{B}:P \to \mathcal{M}(P)$ of the 
barycenter map \eqref{barymap}, define the probability measure $d\mathcal{B}_{x}^{N}$ 
by \eqref{convolution}. Then by the Law of Large Numbers, $d\mathcal{B}_{x}^{N}$ tends 
weakly to the Dirac measure $\delta_{x}$ at $x \in P$. Furthermore, this convergence is uniform 
on $P$. See Subsection \ref{generalS}. We also note that, since the measures $d\mathcal{B}_{x}^{N}$ 
is defined as a dilated convolution powers, it satisfies, for example, the Central Limit Theorem 
(with a suitable dilation) and a large deviations principle at least for fixed $x$. 
In particular, the rate functions for the large deviations principle for the measure $d\mathcal{B}_{x}^{N}$ are 
given in Subsection \ref{LargeDP}, Proposition \ref{LDP}. 

The measure $d\mathcal{B}_{x}^{N}$ is also a section of the barycenter map,  
that is the barycenter of $d\mathcal{B}_{x}^{N}$ is $x$. 
But, it might not be necessary to define $d\mathcal{B}_{x}^{N}$ as a dilated convolution power. 
One might need only the properties that the measure $d\mathcal{B}_{x}^{N}$ has some regularity in $x \in P^{o}$ 
and its barycenter suitably converges to $x$. 
In fact, the measure defining the Bergman-Bernstein approximation in \cite{Z} 
is not a section of the barycenter map, but it has these properties. 
However, as in Section \ref{property}, 
that the measure $d\mathcal{B}_{x}^{N}$ is a dilated convolution power makes analysis much easier 
than something which satisfies only the above properties. 
This is one of the main differences between the Bernstein approximation 
in this paper and the Bergman-Bernstein approximation. See Section \ref{ToricGeometry} for details. }
\end{remark}

Next, we give some examples of Bernstein measures. 

\begin{example}
\label{example1}
Let $P$ be an $m$-dimensional standard simplex, 
\[
P=\{x=(x_{1},\ldots,x_{m}) \in \mathbb{R}^{m}\,;\, 
x_{j} \geq 0,\ \|x\|:=\sum_{j=1}^{m}x_{j} \leq 1\}. 
\]
Let $S=\{e_{0}:=0,e_{1},\ldots,e_{m}\}=P \cap \mathbb{Z}^{m}$ 
where $\{e_{j}\}_{j=1}^{m}$ is the standard basis of $\mathbb{R}^{m}$. 
Define the functions $m_{e_{j}} \in C^{\infty}(P)$ $(j=0,1,\ldots,m)$ by 
\[
m_{e_{0}}(x)=1-\|x\|,\ m_{e_{j}}(x)=x_{j},\ j=1,\ldots,m. 
\]
Then, it is easy to see that the measure $\mathcal{B}(x)=\sum_{j=0}^{m}m_{e_{j}}(x)\delta_{e_{j}}$ 
defines a Bernstein measure with the Bernstein approximation given by \eqref{simplex}. 
The defining matrix $K(x)$, $x \in P^{o}$, is given by (\cite{TZ})
\[
K(x)=
\left(
\frac{\delta_{ij}}{x_{j}}+\frac{1}{1-\|x\|}
\right)_{ij}, 
\]
which is the inverse of the matrix
\[
A(x)=
\left(
x_{j}\delta_{ij}-x_{i}x_{j}
\right)_{ij}. 
\]
\end{example}

\begin{example}
\label{example2}
If $P$ and $Q$ are polytopes in $\mathbb{R}^{m}$ with $P^{o} \neq \emptyset$, $Q^{o} \neq \emptyset$, 
then, their product $P \times Q$ is a polytope in $\mathbb{R}^{2m}$. 
Let $d\mathcal{B}^{P}_{x}$ and $d\mathcal{B}^{Q}_{y}$ be Bernstein measures on $P$ and $Q$, respectively. 
Then, the product measure $d\mathcal{B}_{(x,y)}:=d\mathcal{B}^{P}_{x}d\mathcal{B}_{y}^{Q}$ is a Bernstein 
measure on $P \times Q$. For example, let $\mathcal{B}^{1}:[0,1] \to \mathcal{M}([0,1])$ be the Bernstein measure 
given in the previous example (with $m=1$). Then the Bernstein approximation associated to the Bernstein measure 
\[
d\mathcal{B}_{(x_{1},\ldots,x_{m})}:=d\mathcal{B}^{1}_{x_{1}}\cdots d\mathcal{B}^{1}_{x_{m}}, 
\quad (x_{1},\ldots,x_{m}) \in [0,1]^{m}
\]
on $[0,1]^{m}$ is given by 
\[
\begin{split}
B_{N}(f)&(x_{1},\ldots,x_{m}) \\
=&
\sum_{k_{1},\ldots,k_{m}=0}^{N}f(k_{1}/N,\ldots,k_{m}/N)\prod_{i=1}^{m}
{N \choose k_{i}}x_{i}^{k_{i}}(1-x_{i})^{N-k_{i}}. 
\end{split}
\]
This is a well-known generalization of the original Bernstein measure (\cite{Lo}). 
\end{example}

There is an example of a Bernstein measure which has a smooth density. 
We give such an example only on the unit interval $[0,1]$ as a proposition. 
We define a function $\mu=\mu (\tau)$ on $\mathbb{R}$ by 
\begin{equation}
\label{Todd}
\mu(\tau)=\frac{{\rm Todd}(\tau)-1}{\tau}=\frac{1}{1-e^{-\tau}}-\frac{1}{\tau}. 
\end{equation}
It is easy to show that the function $\mu$ in \eqref{Todd} defines a diffeomorphism 
between $\mathbb{R}$ and $(0,1)$. Let $\tau=\tau(x):(0,1) \to \mathbb{R}$ be 
the inverse function of $\mu$. 
Define a smooth function $\delta(x)$ on $(0,1)$ by 
\begin{equation}
\label{Todddelta}
\delta(x):=\log \chi(\tau(x)) -x\tau(x), \quad 
\chi(\tau)=(e^{\tau}-1)/\tau. 
\end{equation}
Finally, we define a smooth function $\rho (z,x)$ on $[0,1] \times (0,1)$ by 
\begin{equation}
\label{Bsmooth1}
\rho(z,x)=e^{-\delta(x)+(z-x)\tau(x)}=\frac{e^{z\tau(x)}}{\chi(\tau(x))}=\frac{e^{z\tau(x)}\tau(x)}{e^{\tau(x)}-1}. 
\end{equation}

\begin{proposition}
\label{smoothTodd}
For any $x \in [0,1]$, define the measure $d\mathcal{B}_{x}$ by 
\begin{equation}
\label{Bsmooth2}
d\mathcal{B}_{x}(z):=\rho(z,x)\,dz\ (x \in (0,1)),\quad 
d\mathcal{B}_{1}(z):=\delta_{1},\quad d\mathcal{B}_{0}(z):=\delta_{0}. 
\end{equation}
Then the map $\mathcal{B}:[0,1] \to \mathcal{M}([0,1])$ defined by $\mathcal{B}(x)=d\mathcal{B}_{x}$ 
is a Bernstein measure on $[0,1]$. 
\end{proposition}

\begin{proof}
Let us find a Bernstein measure $d\mathcal{B}_{x}$ on the unit interval $[0,1]$ which has the form 
\begin{equation}
\label{smoothm}
d\mathcal{B}_{x}(z)=\rho(z,x)\,dz,\quad z \in [0,1],\quad x \in (0,1), 
\end{equation}
where $\rho (z,x)$ is a smooth on $[0,1] \times (0,1)$. If the measure of the form \eqref{smoothm} 
with a smooth function $\rho$ is a Bernstein measure, then the density function $\rho$ must satisfy the following: 
\begin{equation}
\label{condB}
\int_{0}^{1}\rho (z,x)\,dz=1,\quad 
\int_{0}^{1}z\rho(z,x)\,dz=x,\quad 
\partial_{x}\log \rho (z,x)=K(x)(z-x), 
\end{equation}
where $K(x)$ is a positive smooth function on $(0,1)$ (the defining matrix of $\mathcal{B}$) 
which we need to specify. 
We take a potential function $\phi$ of $K$ on $(0,1)$, that is $\phi''=K$.
Inserting $K=\phi''$ for the third equation in \eqref{condB} and solving it as a differential equation for $\rho$, 
we have
\begin{equation}
\label{condB3}
\rho(z,x)=C(z)e^{\phi(x)+\phi'(x)(z-x)}
\end{equation}
with a positive function $C(z)$. 
For simplicity, we assume that $C(z)$ is constant, $C(z) \equiv C$. 
Substituting this for the first and second equation in \eqref{condB}, we have 
\[
\phi'=Ce^{-x\phi'+\phi}(e^{\phi'}-1),\quad 
x(\phi')^{2}=Ce^{-x\phi'+\phi}(\phi'e^{\phi'}-e^{\phi'}+1). 
\]
Eliminating the term $Ce^{-x\phi'+\phi}$ from the above, we get 
\begin{equation}
\label{condB4}
x\tau(e^{\tau}-1)=\tau e^{\tau}-e^{\tau}+1,\quad \tau(x)=\phi'(x).  
\end{equation}
Note that the defining matrix (function) $K$ must be positive on $(0,1)$ (see Lemma \ref{InverseMatrix}), 
and hence $x \mapsto \tau(x)$ is a diffeomorphism from $(0,1)$ onto its image. 
Denote its inverse function $\mu=\mu(\tau)$. 
From \eqref{condB4}, the function $\mu(\tau)$ must be given by \eqref{Todd}. 
The defining function $K(x)$ must be given by $K(x)=1/\mu'(\tau(x))=\tau'(x)$, $x \in (0,1)$. 
Therefore, the defining function $K(x)$ and its 
integral $\tau(x)$ are determined uniquely by the equation \eqref{condB}. 
Clearly $\mu(\tau)=(\log \chi (\tau))'$ where the function $\chi$ is defined in \eqref{Todddelta}. 
Then, we have $\delta'(x)=-\tau(x)$ and hence we can take $\phi=-\delta$. 
From this choice, we have $C=1$ and hence the function $\rho (z,x)$ must be the form \eqref{Bsmooth1}. 
One can easily check that, this function $\rho$ actually satisfies the equations \eqref{condB}. 
Note that $A(x):=K(x)^{-1}$ is continuously extended to $[0,1]$ with $A(0)=A(1)=0$. 
But, the function $\rho(z,x)$ can not be extended to $x=0,1$ continuously. 
However, for $x \in [0,1]$, define a probability measure $d\mathcal{B}_{x}$ on $[0,1]$ by \eqref{Bsmooth2}. 
For any $f \in C([0,1])$ and $x \in [0,1]$, we set $B(f)(x)=\int_{0}^{1}f(z)\,d\mathcal{B}_{x}(z)$ which is smooth on $(0,1)$. 
If $f \in C^{1}([0,1])$, an integration by parts shows that $B(f)(x) \to f(1)$ when $x \to 1$ and $B(f)(x) \to f(0)$ when $x \to 0$. 
Since $\sup_{0 < x <1}|B(f)(x)-B(g)(x)| \leq \|f-g\|_{C([0,1])}$, $f,g \in C([0,1])$,
$B(f)$ is continuous on $[0,1]$ for any $f \in C([0,1])$. 
Hence $d\mathcal{B}_{x}$ is a Bernstein measure on $[0,1]$. 
By the above discussion, one knows that 
a Bernstein measure $d\mathcal{B}_{x}$ on $[0,1]$ having a smooth density is uniquely determined, 
and is given by \eqref{Bsmooth2}, \eqref{Bsmooth1}. 
\end{proof}

One of our main theorems is the following, which generalizes the asymptotic expansion 
of the Bernstein polynomials on standard simplices given by Abel-Ivan \cite{AI} and H\"{o}rmander \cite{Ho2}. 

\begin{theorem}
\label{AsympExp}
Let $\mathcal{B}:P \to \mathcal{M}(P)$ be a Bernstein measure on the polytope $P$. 
Then, for each non-negative integer $\nu$, there exists a differential operator $L_{\nu}(x,\partial)$ of 
order $2\nu$ such that, for any $f \in C^{\infty}(P)$, we have the following asymptotic expansion: 
\begin{equation}
\label{expansion}
B_{N}(f) \sim \sum_{\nu \geq 0}N^{-\nu}L_{\nu}(x,\partial)f, 
\end{equation}
where the expansion holds uniformly on $P$. 
$L_{0}(x,\partial)$ and $L_{1}(x,\partial)$ are given by the following: 
\begin{equation}
\label{0and1}
L_{0}(x,\partial)f=f(x),\quad 
L_{1}(x,\partial)f=\frac{1}{2}{\rm Tr}(A(x)\nabla^{2}f(x)), 
\end{equation}
where $\nabla^{2}f$ is the Hessian of $f$ and $A(x) \in {\rm Sym}(m,\mathbb{R})$ 
denotes the inverse of the defining matrix $K(x) \in {\rm Sym}(m,\mathbb{R})$. 
The asymptotic expansion \eqref{expansion} can be differentiated any number of times 
and the resulting expansion holds locally uniformly on $P^{o}$. 
\end{theorem}

The matrix $K(x)$ is not assumed to be non-degenerate. However, 
it follows that $K(x)$ is non-degenerate with the inverse matrix $A(x)$ given by 
\begin{equation}
\label{matrixA}
A(x)=\int_{P}(z-x)\otimes (z-x)\,d\mathcal{B}_{x}(z),\quad x \in P^{o}. 
\end{equation}
Note that, though the defining matrix $K(x)$ is defined only for $x$ in $P^{o}$, 
the matrix $A(x)$ in \eqref{matrixA} is defined and continuous on whole $P$. 
We have a computable representation for the differential operators $L_{\nu}(x,\partial)$. 
See Section \ref{property} for details. 

Next, to explain a construction of finitely supported Bernstein measures, let us prepare some notation. 
Let $S \subset \mathbb{R}^{m}$ be a finite set whose convex hull is $P$. Since $P$ is assumed to have 
non-empty interior, the finite set $S$ satisfies the following: 
\begin{equation}
\label{nonvoid}
{\rm span}_{\mathbb{R}}\{\alpha -\beta\,;\,\alpha,\beta \in S\}=\mathbb{R}^{m}. 
\end{equation}
Fix a positive function $c:S\to \mathbb{R}_{>0}$ on $S$. 
Define the map $\mu_{S,c}:\mathbb{R}^{m} \to \mathbb{R}^{m}$ by 
\begin{equation}
\label{momentM}
\mu_{S,c}(\tau):=\sum_{\alpha \in S}\frac{c(\alpha)e^{\ispa{\alpha,\tau}}}
{\sum_{\beta \in S}c(\beta)e^{\ispa{\beta,\tau}}}\alpha,\quad 
\tau \in \mathbb{R}^{m},  
\end{equation}
where $\ispa{u,v}$ denotes the usual Euclidean inner product on $\mathbb{R}^{m}$. 
It is well-known (\cite{Fu}) that $\mu_{S,c}$ is a diffeomorphism from $\mathbb{R}^{m}$ to $P^{o}$. 
Its inverse is denoted by $\tau_{S,c}:P^{o} \to \mathbb{R}^{m}$. 
For each $\alpha \in S$, we define the function $m_{S,c,\alpha}$ by 
\begin{equation}
\label{Scfunct}
m_{S,c,\alpha}(x)=\frac{c(\alpha)e^{\ispa{\alpha,\tau_{S,c}(x)}}}
{\sum_{\beta \in S}c(\beta)e^{\ispa{\beta,\tau_{S,c}(x)}}},\quad x \in P^{o}. 
\end{equation}
Then, clearly, $m_{S,c,\alpha}$ is a positive smooth function on $P^{o}$. Furthermore, it is not hard 
to show that $m_{S,c,\alpha}$ can continuously be extended to the boundary of $P$ so that 
it defines a continuous function on $P$ which is smooth on $P^{o}$ (Lemma \ref{contimS}). Then, 
since $\tau_{S,c}$ is the inverse map of $\mu_{S,c}$, these functions satisfy
\begin{equation}
\label{bsect}
\sum_{\alpha \in S}m_{S,c,\alpha}(x)\alpha =x
\end{equation}
for each $x \in P^{o}$, and by the continuity, it holds for each $x \in P$. 

\begin{theorem}
\label{char}
Let $S \subset P$ and $c:S \to \mathbb{R}_{>0}$ be as above. 
Let $m_{S,c,\alpha}$, $\alpha \in S$ be the functions defined by \eqref{Scfunct}. 
Then, the section $\mathcal{B}_{S,c}:P \to \mathcal{M}(P)$ defined by 
\begin{equation}
\label{fB}
\mathcal{B}_{S,c}(x)=\sum_{\alpha \in S}m_{S,c,\alpha}(x)\delta_{\alpha}
\end{equation}
is a finitely supported Bernstein measure. 
Conversely, suppose that there is a finitely supported Bernstein measure $\mathcal{B}$ 
with support $S$. Then, there exists a weight function $c:S \to \mathbb{R}_{>0}$ 
such that $\mathcal{B}=\mathcal{B}_{S,c}$. 
\end{theorem}

We discuss some properties of Bernstein measures and 
give proofs of Theorems \ref{AsympExp}, \ref{char} in Section \ref{property}. 
As we pointed out in Remark \ref{remark1}, the measures $d\mathcal{B}_{x}^{N}$ 
satisfy the large deviations principle. 
In particular, we give a concrete description of the rate functions for the large deviations in Subsection \ref{LargeDP}. 
In Section \ref{ToricGeometry}, we discuss a difference between the Bernstein measures 
and the Bergman-Bernstein measures defining the Bergman-Bernstein approximations introduced in \cite{Z}. 
In particular, we give some conditions in Proposition \ref{BBversusB} 
for when Zelditch's Bergman-Bernstein measures coincide with our Bernstein measures.

\subsection{General properties of sections of the barycenter map}
\label{generalS}

Before proceeding to the discussion on Bernstein measures, we give 
some accounts on properties of general continuous sections of the barycenter map 
$b:\mathcal{M}(P) \to P$ on the polytope $P$. 
Throughout the paper, $\{e_{j}\}_{j=1}^{m}$ denotes the standard basis for $\mathbb{R}^{m}$. 

\begin{lemma}
\label{general1}
Suppose that we are given a sequence of bounded linear maps $B_{N}:C(P) \to C(P)$ 
satisfying the following: 
\begin{enumerate}
\item $B_{N}(1)=1$, $B_{N}(f) \geq 0$ for non-negative continuous functions $f$;  
\item $B_{N}(x^{\alpha}) \to x^{\alpha}$ uniformly as $N \to \infty$ 
for $\alpha \in \mathbb{Z}_{\geq 0}^{m}$ with $\|\alpha\| \leq 2$. 
\end{enumerate}
Then, $B_{N}(f) \to f$ uniformly as $N \to \infty$. 
\end{lemma}

\begin{proof}
First, we note that the following holds: 
\begin{itemize}
\item $B_{N}(f)$ is real-valued if $f$ is real; 
\item $|B_{N}(f)| \leq B_{N}(|f|)$, $f \in C(P)$;  
\item $\|B_{N}(f)\|_{C(P)} \leq \|f\|_{C(P)}$, $f \in C(P)$. 
\end{itemize}
Let $f \in C^{\infty}(P)$. We fix $x \in P$. 
For any multi-index $\alpha \in \mathbb{Z}_{\geq 0}^{m}$, we set $\chi_{\alpha}(z)=z^{\alpha}$. 
Then, inserting the Taylor expansion 
\[
\begin{gathered}
f(z)=f(x)+R_{1,x}(z)+R_{2,x}(z), \\
R_{1,x}(z)=\ispa{\nabla f (x),z-x}=\sum_{j=1}^{m}{\textstyle \frac{\partial f}{\partial x_{j}}}(x)(z_{j}-x_{j}),\\
R_{2,x}(z)=\int_{0}^{1}(1-t)\ispa{\nabla^{2}f(x+t(z-x)) (z-x),(z-x)}\,dt
\end{gathered}
\]
around $x \in P$ to $B_{N}(f)$, we have $B_{N}(f)-f=B_{N}(R_{1,x})+B_{N}(R_{2,x})$. 
By assumption, we have 
$B_{N}(R_{1,x})(x)=\sum_{j=1}^{m}\frac{\partial f}{\partial x_{j}}(x)(B_{N}(\chi_{e_{j}})(x)-x_{j}) \to 0$ uniformly in $x \in P$. 
Now, we have $|R_{2,x}(z)| \leq \|f\|_{C^{2}(P)}|z-x|^{2}$ and hence 
\[
|B_{N}(R_{2,x})(x)| \leq \|f\|_{C^{2}(P)}\sum_{j=1}^{m}
(B_{N}(\chi_{2e_{j}})(x)-2x_{j}B_{N}(\chi_{e_{j}})(x)+x_{j}^{2}). 
\]
Since $B_{N}(\chi_{2e_{j}})$ and $B_{N}(\chi_{e_{j}})$ converge uniformly to $x_{j}^{2}$, $x_{j}$, respectively, 
we conclude that $B_{N}(f) \to f$ uniformly on $P$ for any $f \in C^{\infty}(P)$. 
Since $C^{\infty}(P)$ is dense in $C(P)$, we conclude the assertion. 
\end{proof}

\begin{lemma}
\label{general2}
Let $\mathcal{B}:P \to \mathcal{M}(P)$ be a continuous section of $b:\mathcal{M}(P) \to P$. 
For positive integer $N$, we define the measure $d\mathcal{B}_{x}^{N}$ by \eqref{convolution} 
and, for $f \in C(P)$, the function $B_{N}(f) \in C(P)$ by \eqref{BNf}. 
Then, $B_{N}(f)$ converges uniformly to $f$ as $N \to \infty$ for all $f \in C(P)$. 
\end{lemma}

The following proof using Lemma \ref{general1} is pointed out to the author by professor Bando. 

\begin{proof}
Note that we have $b(d\mathcal{B}_{x}^{N})=x$. Thus, by Lemma \ref{general1}, we only need to show 
that $B_{N}(\chi_{\alpha})$ converges to $\chi_{\alpha}$ for each $\alpha$ with $\|\alpha\|=2$. 
Then, by definition of the convolution powers \eqref{remark4}, we have 
\[
\begin{split}
B_{N}(\chi_{e_{i}+e_{j}})&=\frac{1}{N^{2}}
\sum_{k,l=1}^{N}\int_{P \times \cdots \times P}\!\!\!
z_{i}^{(k)}z_{j}^{(l)}\,d\mathcal{B}_{x}(z^{(1)})\cdots d\mathcal{B}_{x}(z^{(N)}) \\
&=\frac{N(N-1)}{N^{2}}x_{i}x_{j} +\frac{1}{N^{2}}
\sum_{k=1}^{N}\int_{P \times \cdots \times P}\!\!\!
z_{i}^{(k)}z_{j}^{(k)}\,d\mathcal{B}_{x}(z^{(1)})\cdots d\mathcal{B}_{x}(z^{(N)}). 
\end{split}
\]
Since the second term in the above is $O(1/N)$, we conclude the assertion. 
\end{proof}

By using Lemmas \ref{general1}, \ref{general2}, we further discuss as follows. 
Let $\mathcal{B}:P \to \mathcal{M}(P)$ be a continuous section of $b:\mathcal{M}(P) \to P$. 
For a positive integer $N$, we define the measure $d\mathcal{B}_{x}^{N}$ by \eqref{convolution} 
and for $f \in C(P)$ the function $B_{N}(f) \in C(P)$ by \eqref{BNf}. 
For any multi-index $\alpha \in \mathbb{Z}_{\geq 0}^{m}$, we define the function $I_{N,\alpha}(x)$ by 
\begin{equation}
\label{functI}
I_{N,\alpha}(x):=N^{\|\alpha\|}\int_{P}(z-x)^{\alpha}\,d\mathcal{B}_{x}^{N}(z). 
\end{equation}
In the definition \eqref{functI} of the functions $I_{N,\alpha}$, the integral is multiplied by $N^{\|\alpha\|}$. 
The reason for this would be clarified in Section \ref{property}. 
Now, for any $f \in C^{\infty}(P)$, substituting the Taylor expansion
\begin{equation}
\label{TaylorE}
\begin{gathered}
f(z)=\sum_{\|\alpha\| \leq 2n-1}\frac{f^{(\alpha)}(x)}{\alpha!}(z-x)^{\alpha}
+\sum_{\|\alpha\|=2n}\frac{R_{2n}^{\alpha}(z,x)}{\alpha!}(z-x)^{\alpha},\\
R_{2n}^{\alpha}(z,x)=2n\int_{0}^{1}(1-t)^{2n-1}f^{(\alpha)}(x+t(z-x))\,dt, 
\end{gathered}
\end{equation}
we have 
\begin{equation}
\label{approxBN}
B_{N}(f)(x)=\sum_{\|\alpha\| \leq 2n-1}\frac{I_{N,\alpha}(x)}{\alpha! N^{\|\alpha\|}} \partial^{\alpha}f (x)
+S_{2n,N}(x), 
\end{equation}
where the function $S_{2n,N}(x)$ is given by 
\begin{equation}
\label{error}
S_{2n,N}(x)=\sum_{\|\alpha\| =2n}\frac{1}{\alpha!}\int_{P}
R_{2n}^{\alpha}(z,x)(z-x)^{\alpha}\,d\mathcal{B}_{x}^{N}(z). 
\end{equation}
By Lemma \ref{general2}, $\|I_{N,\alpha}\|_{C(P)}=o(N^{\|\alpha\|})$. 
Therefore, we have 
\begin{equation}
\label{esterror}
\begin{split}
|S_{2n,N}(x)| & \leq \|f\|_{C^{2n}(P)}\sum_{\|\alpha\| =2n}
\frac{1}{\alpha!}\int_{P}|(z-x)^{\alpha}|\,d\mathcal{B}_{x}^{N}(z) \\
& \leq C \|f\|_{C^{2n}(P)}N^{-2n}\sum_{\|\beta\|=n}\frac{1}{\beta!}I_{N,2\beta}(x), 
\end{split}
\end{equation}
which is clearly of order $o(1)$ as $N \to \infty$ uniformly in $x \in P$.
However we note that, in the expression \eqref{approxBN}, each term with $\|\alpha\| \geq 1$ is of order $o(1)$. 
Thus, one need to find asymptotic behavior of the functions $I_{N,\alpha}$ to make 
\eqref{approxBN} an asymptotic expansion.

\section{Properties of Bernstein measures; proofs of main theorems}
\label{property}

In this section, we discuss some properties of Bernstein measures defined in the previous section, 
and give proofs of main Theorems \ref{AsympExp}, \ref{char}.

\subsection{Properties of Bernstein measures}
\label{proofofM1}
Let $\mathcal{B}:P \to \mathcal{M}(P)$ be a Bernstein measure, and, for any $f \in C(P)$ and positive integer $N$, 
let $B_{N}(f)$ be the Bernstein approximation defined by \eqref{BNf}. 
The corresponding defining matrix is denoted by 
$K:P^{o} \to {\rm Sym}(m,\mathbb{R})$.

\begin{lemma}
\label{InverseMatrix}
For each $x \in P^{o}$, the linear map $A(x):\mathbb{R}^{m} \to \mathbb{R}^{m}$ defined by \eqref{matrixA} is the inverse of the 
symmetric linear map of $K(x):\mathbb{R}^{m} \to \mathbb{R}^{m}$. 
Furthermore, $A(x)$, and hence $K(x)$ is positive definite for each $x \in P^{o}$. 
\end{lemma}

\begin{proof}
Differentiating the identity $x=\int_{P} z\,d\mathcal{B}_{x}(z)$, 
we find that, for any $u \in \mathbb{R}^{m}$, 
\[
u=\int_{P} \ispa{K(x)(z-x),u}z\,d\mathcal{B}_{x}(z)=
\int_{P} \ispa{K(x)(z-x),u}(z-x)\,d\mathcal{B}_{x}(z),  
\]
which equals $A(x)K(x)u$ by definition \eqref{matrixA} of $A(x)$. 
The fact that $A(x)$ is non-negative (and hence positive definite) follows from the definition \eqref{matrixA}. 
\end{proof}

\begin{lemma}
\label{diffBNf}
For any $f \in C(P)$, we have 
\[
\nabla B_{N}(f) (x)=N\int_{P}f(z)K(x)(z-x)\,d\mathcal{B}_{x}^{N}. 
\]
\end{lemma}

\begin{proof}
The assertion is shown by computing
$\nabla B_{N}(f)(x)$ with the definition of the dilated convolution power \eqref{remark4} in the following way: 
\[
\begin{split}
\nabla B_{N}(f)(x)&=\sum_{j=1}^{N}\int_{P\times \cdots \times P}\!\!\!\!\! f((z_{1}+\cdots +z_{N})/N)
K(x)(z_{j}-x)\,d\mathcal{B}_{x}(z_{1})\cdots d\mathcal{B}_{x}(z_{N}) \\
&=\int_{P \times \cdots \times P}
f((z_{1}+\cdots +z_{N})/N) \times \\ 
&\hspace{3cm}\times 
K(x)(z_{1}+ \cdots +z_{N}-Nx)\,d\mathcal{B}_{x}(z_{1})\cdots d\mathcal{B}_{x}(z_{N})\\
&=\int_{NP}f(w/N)K(x)(w-Nx)\,d(\mathcal{B}_{x} \ast \cdots \ast \mathcal{B}_{x})(w)\\
&=N\int_{P}f(z)K(x)(z-x)\,d\mathcal{B}_{x}^{N}(z). 
\end{split}
\]
\end{proof}

\begin{lemma}
\label{repofA}
We have $\displaystyle A(x)=N\int_{P}(z-x) \otimes (z-x)\,d\mathcal{B}_{x}^{N}(z)$. 
\end{lemma}

\begin{proof}
Since the barycenter of $d\mathcal{B}_{x}^{N}$ is also $x$, we have $x=\int_{P} z\,d\mathcal{B}_{x}^{N}(z)$. 
Differentiating this identity and using Lemma \ref{diffBNf}, we obtain 
\[
{\rm Id}=N\int_{P} z \otimes K(x)(z-x)\,d\mathcal{B}_{x}^{N}(z)
=N \int_{P} (z-x) \otimes K(x)(z-x)\,d\mathcal{B}_{x}^{N}(z), 
\]
which means the desired formula. 
\end{proof}

For any $u \in \mathbb{R}^{m}$, we define the first order differential operator $D_{u}$ on $P$ by 
\begin{equation}
\label{Adiff}
D_{u}f (x):=\ispa{A(x)\nabla f (x),u}=\ispa{\nabla f(x),A(x)u}, 
\end{equation}
where $\nabla f (x)$ is the gradient of a function $f$ on $P$. 
For any smooth function $L:P^{o} \to {\rm Sym}(m,\mathbb{R})$, we write 
\[
[\nabla L(x)v]u:=\left. 
\frac{d}{dt}
\right|_{t=0}L(x+tv)u,\quad x \in P^{o},\ u,v \in \mathbb{R}^{m}. 
\]

\begin{lemma}
\label{commute1}
For any $u, v \in \mathbb{R}^{m}$, we have $[D_{u},D_{v}]=0$. 
\end{lemma}

\begin{proof}
For any $\lambda \in \mathbb{R}^{m}$, a simple computation shows 
\begin{equation}
\label{aux1}
\begin{gathered}
\ispa{\nabla D_{u}f(x),\lambda}=\ispa{\nabla^{2}f(x)\lambda,A(x)u}+\ispa{\nabla f(x),[\nabla A(x)\lambda]u}, \\
[\nabla A(x)\lambda]u=\int_{P} \ispa{z-x,u}\ispa{K(x)(z-x),\lambda}(z-x)\,d\mathcal{B}_{x}(z), 
\end{gathered}
\end{equation}
where $\nabla^{2}f(x)$ is the Hessian of $f$. 
Thus, by setting $\lambda=A(x)v$, we find 
\[
\begin{split}
D_{v}&D_{u}f(x) \\
&=\ispa{\nabla^{2}f(x)A(x)v,A(x)u}+
\int_{P} \ispa{z-x,u}\ispa{z-x,v}\ispa{\nabla f(x),z-x}\,d\mathcal{B}_{x}(z), 
\end{split}
\]
which is obviously symmetric in $u$ and $v$. 
\end{proof}

The following lemma shows an integrability of the 
defining matrix $K$, which will be used to prove Theorem \ref{char}. 

\begin{lemma}
\label{commute2}
We have $[\nabla K(x)u]v=[\nabla K(x)v]u$ 
for any $u,v \in \mathbb{R}^{m}$ and $x \in P^{o}$. 
\end{lemma}

\begin{proof}
For $x \in P^{o}$, we have $K(x)=A(x)^{-1}$, and hence, by \eqref{aux1}, we obtain
\[
\begin{split}
[\nabla K(x)u]v&=-K(x)[\nabla A(x)u]K(x)v\\
&=-\int_{P}\ispa{K(x)(z-x),v}\ispa{K(x)(z-x),u}K(x)(z-x)\,d\mathcal{B}_{x}(z), 
\end{split}
\]
which is clearly symmetric in $u$ and $v$. 
\end{proof}

The next proposition shows a uniqueness of the Bernstein measure 
for a given function $K:P^{o} \to {\rm Sym}(m,\mathbb{R})$. 
For any multi-index $\alpha=(\alpha_{1},\ldots,\alpha_{m}) \in \mathbb{Z}_{\geq 0}^{m}$, 
we write 
\begin{equation}
\label{Adiff2}
D^{\alpha}:=D_{1}^{\alpha_{1}}\cdots D_{m}^{\alpha_{m}},\quad 
D_{j}:=D_{e_{j}}, 
\end{equation}
where $D_{e_{j}}$ is defined in \eqref{Adiff}.

\begin{proposition}
\label{uniqueness}
Let $K:P^{o} \to {\rm Sym}(m,\mathbb{R}^{m})$ be a smooth map. 
Then, the Bernstein measure with the defining matrix $K$ is, 
if it exists, unique. 
\end{proposition}

\begin{proof}
Let $\mathcal{B}:P\to \mathcal{M}(P)$ be a Bernstein measure with the defining 
matrix $K:P^{o} \to {\rm Sym}(m,\mathbb{R})$. Then, the matrix $K(x)$ is non-degenerate and the inverse matrix 
is given by $A:P \to {\rm Sym}(m,\mathbb{R})$. 
We note that the inverse matrix $A(x)$ defines the differential operators 
$L_{j}:=X_{j}+D_{j}$, $j=1,\ldots,m$, where $X_{j}$ is the multiplication operator: $X_{j}f (x)=x_{j}f(x)$. 
Then, it is easy to see that 
\begin{equation}
\label{creation1}
L_{j}B(f)(x)=B(X_{j}f),\quad 
j=1,\ldots,m, 
\end{equation}
where $B(f)$ is defined in \eqref{integralB}. 
For each $\alpha \in \mathbb{Z}_{\geq 0}^{m}$, we denote the monomial with 
the weight $\alpha$ by $\chi_{\alpha}(z)=z^{\alpha}$. 
Then, by \eqref{creation1}, we have $L_{j}B(\chi_{\alpha})=B(\chi_{\alpha +e_{j}})$. 
A direct computation shows that 
\begin{equation}
\label{commutator1}
D_{k}X_{j}-X_{j}D_{k}=\ispa{A(x)e_{j},e_{k}}=I_{1,e_{j}+e_{k}}, 
\end{equation}
and which is symmetric in $j$ and $k$. Hence $[L_{j},L_{k}]=0$. 
We denote $L^{\alpha}=L_{1}^{\alpha_{1}}\cdots L_{m}^{\alpha_{m}}$ 
for $\alpha=(\alpha_{1},\ldots,\alpha_{m}) \in \mathbb{Z}_{\geq 0}^{m}$. 
We then obtain 
\begin{equation}
\label{creation2}
B(\chi_{\alpha})=L^{\alpha}B(1)=L^{\alpha} \cdot 1
\end{equation}
for each $\alpha \in \mathbb{Z}_{\geq 0}^{m}$. 
By \eqref{creation2}, the restriction of the map $B:C(P) \to C(P)$ to the 
set of polynomials on $P$ is determined by the matrix $K(x)$ because 
the differential operators $L_{j}$ are defined in terms of $A(x)=K(x)^{-1}$. 
Since $B:C(P) \to C(P)$ is continuous, it is determined by $K(x)$. 
\end{proof}

Now, for each $\alpha \in \mathbb{Z}_{\geq 0}^{m}$ and positive integer $N$, define 
the function $I_{N,\alpha} \in C(P) \cap C^{\infty}(P^{o})$ by \eqref{functI} with respect to the measure 
$d\mathcal{B}_{x}^{N}$ determined by the Bernstein measure $\mathcal{B}:P \to \mathcal{M}(P)$. 
Note that, by definition, 
\[
I_{1,e_{i}+e_{j}}(x)=\ispa{A(x)e_{i},e_{j}}=\int_{P}(z_{i}-x_{i})(z_{j}-x_{j})\,d\mathcal{B}_{x}(z), 
\]
where $z_{j}$ denotes the $j$-th coordinate of $z \in \mathbb{R}^{m}$. 
From this, we have $I_{1,e_{i}+e_{j}} \in C(P) \cap C^{\infty}(P^{o})$. 
\begin{lemma}
\label{recursion}
The following identities hold; 
\begin{enumerate}
\item $I_{N,0}(x) \equiv 1$, $I_{N,e_{j}}(x) \equiv 0$, 
$I_{N,e_{i}+e_{j}}(x)=N\ispa{A(x)e_{i},e_{j}}$, \\
$I_{N,e_{i}+e_{j}+e_{k}}(x)=NI_{1,e_{i}+e_{j}+e_{k}}(x)$ for any $1 \leq i,j,k \leq m$ and $N$.  
\item For any $\alpha \in \mathbb{Z}_{\geq 0}^{M}$, $1 \leq j \leq m$ and $N$, 
\begin{equation}
\label{requr}
I_{N,\alpha +e_{j}}(x)=D_{j}I_{N,\alpha}(x)+
\sum_{i=1}^{m}\alpha_{i}I_{N,\alpha-e_{i}}(x)I_{N,e_{i}+e_{j}}(x). 
\end{equation}
\end{enumerate}
\end{lemma}

\begin{proof}
First two of (1) are obvious by definition \eqref{functI}. 
The third formula of (1) follows from Lemma \ref{repofA}. 
The fourth formula in (1) and 
(2) follow from direct computations. 
\end{proof}

By Lemma \ref{recursion}, we have $D^{\beta}I_{N,\alpha} \in C(P) \cap C^{\infty}(P^{o})$ 
for any positive integer $N$ and multi-indices $\alpha$, $\beta$. Furthermore, we have the following. 
\begin{lemma}
\label{polyn}
For any $\alpha \in \mathbb{Z}_{\geq 0}^{m}$ and positive integer $N$, the function $I_{N, \alpha}$ 
is a polynomial in $N$ of degree $[\|\alpha\|/2]$ with coefficients in $C(P) \cap C^{\infty}(P^{o})$. 
More precisely, $I_{N,\alpha}$ is of the form
\begin{equation}
\label{polyN}
I_{N,\alpha}(x)=\sum_{l=0}^{[\|\alpha\|/2]}p_{\alpha,l}(x)N^{l}, 
\end{equation}
where $p_{\alpha,l} \in C(P) \cap C^{\infty}(P^{o})$ and $p_{\alpha,l}$ does not depend on $N$. 
Furthermore, $D^{\beta}p_{\alpha,l} \in C(P) \cap C^{\infty}(P^{o})$ for every $\beta$. 
In particular, we have $I_{N,\alpha}(x)=O(N^{[\|\alpha\|/2]})$ uniformly on $P$. 
\end{lemma}

Lemma \ref{polyn} can easily be shown by induction on $k=\|\alpha\|$ with Lemma \ref{recursion}. 
Note that the functions $p_{\alpha,l}$ can be computed inductively. 
Indeed, by substituting \eqref{polyN} to \eqref{requr}, we have 
\begin{equation}
\label{polyn2}
I_{N,\alpha+e_{j}}(x)=\sum_{l=0}^{[(\|\alpha\|+1)/2]}
\left[
D_{j}p_{\alpha,l}(x)+\sum_{i=1}^{m}\alpha_{i}I_{1,e_{i}+e_{j}}(x)p_{\alpha-e_{i},l-1}(x)
\right]N^{l}. 
\end{equation}

\subsection{Proof of Theorem \ref{AsympExp}}
By using the Taylor expansion of $f$ around $x \in P$, 
the function $B_{N}(f)$ is represented as \eqref{approxBN}, \eqref{error}. 
The error term $S_{2n,N}$ in \eqref{approxBN} is estimated as \eqref{esterror}, 
and hence, by Lemma \ref{polyn}, we have $\|S_{2n,N}\|_{C(P)} =O(\|f\|_{C^{2n}(P)} N^{-n})$. 
Thus we have 
\[
B_{N}(f)(x)=\sum_{\|\alpha\| \leq 2n-1}\frac{f^{(\alpha)}(x)}{\alpha! N^{\|\alpha\|}}I_{N,\alpha}(x) +
O(\|f\|_{C^{2n}(P)}N^{-n}). 
\]
Then, by substituting \eqref{polyN} for the above, we have 
\[
\begin{split}
\sum_{\|\alpha\| \leq 2n-1}&
\frac{f^{(\alpha)}(x)}{\alpha! N^{\|\alpha\|}}
I_{N,\alpha}(x)\\
&=\sum_{\nu=0}^{n-1}\left(
\sum_{l=\nu}^{2\nu}\sum_{\|\alpha\|=l}
\frac{f^{(\alpha)}(x)}{\alpha!} 
p_{\alpha,l-\nu}(x)
\right)N^{-\nu} \\
&\hspace{2cm}+\sum_{\nu=n}^{2n-1}
\left(
\sum_{l=\nu}^{2n-1}
\sum_{\|\alpha\|=l}\frac{f^{(\alpha)}(x)}{\alpha!}
p_{\alpha,l-\nu}
\right)N^{-\nu}\\
&=\sum_{\nu=0}^{n-1}\left(
\sum_{l=\nu}^{2\nu}\sum_{\|\alpha\|=l}
\frac{f^{(\alpha)}(x)}{\alpha!} 
p_{\alpha,l-\nu}(x)
\right)N^{-\nu}
+O(\|f\|_{C^{2n-1}(P)}N^{-n}). 
\end{split}
\]
Therefore, if we set
\begin{equation}
\label{diffop}
L_{\nu}(x,\partial)=\sum_{l=\nu}^{2\nu}\sum_{\|\alpha\|=l}
p_{\alpha,l-\nu}(x)\frac{1}{\alpha!}\partial^{\alpha}, 
\end{equation}
we obtain 
\[
B_{N}(f)(x)=\sum_{\nu=0}^{n-1}N^{-\nu}L_{\nu}(x,\partial)f+O(\|f\|_{C^{2n}(P)}N^{-n}), 
\]
which shows the asymptotic expansion \eqref{expansion}. 
The formulas for the differential operators $L_{0}(x,\partial)$ and $L_{1}(x,\partial)$ are 
obtained from Lemma \ref{recursion} and \eqref{diffop}.  

Next, we show that the expansion \eqref{expansion} can be differentiated any number of times. 
First of all we note that, if $\partial_{j}B_{N}(f)$ has an asymptotic expansion of the form $\sum_{\nu \geq 0}N^{-\nu}g_{j,\nu}(x)$ 
locally uniformly in $x \in P^{o}$ for each $j=1,\ldots,m$, then it is easy to show that $g_{j,\nu}=\partial_{j}L_{\nu}(x,\partial)f$. 
Instead of using the partial differential operators $\partial_{j}$, we use the operators $D_{j}$ defined in \eqref{Adiff2}. 
From Lemma \ref{diffBNf}, we have 
\begin{equation}
\label{commutator2}
D_{j}B_{N}(f)=N(B_{N}(X_{j}f)-X_{j}B_{N}(f)),\quad 
\mbox{that is,}\quad D_{j}B_{N}=N[B_{N},X_{j}]. 
\end{equation}
Therefore, since $L_{0}(x,\partial)=1$, we have 
$D_{j}B_{N}(f) \sim \sum_{\nu \geq 0}N^{-\nu}[L_{\nu+1}(x,\partial),X_{j}]f$. 
This holds uniformly on $P$. From this we also have $D_{j}L_{\nu}(x,\partial)=[L_{\nu+1}(x,\partial),X_{j}]$. 
Next, assume that we have 
\begin{equation}
\label{expD2}
D^{\alpha}B_{N}(f) \sim \sum_{\nu \geq 0}N^{-\nu}D^{\alpha}L_{\nu}(x,\partial)f
\end{equation}
uniformly on $P$ for any $f \in C^{\infty}(P)$ and $\alpha \in \mathbb{Z}_{\geq 0}^{m}$ 
with $\|\alpha\| \leq k$ where $k$ is a fixed positive integer. 
Then, by \eqref{commutator2}, we have 
\[
D_{j}D^{\alpha}B_{N}(f)=D^{\alpha}D_{j}B_{N}(f)=ND^{\alpha}(B_{N}(X_{j}f)-X_{j}B_{N}(f)), 
\]
and hence 
\[
\begin{split}
D_{j}D^{\alpha}B_{N}(f) &\sim \sum_{\nu \geq 0}N^{1-\nu}(D^{\alpha}L_{\nu}(x,\partial)X_{j}-D^{\beta}X_{j}L_{\nu}(x,\partial))f\\
&=\sum_{\nu \geq 0}N^{-\nu}D^{\alpha}[L_{\nu+1}(x,\partial),X_{j}]f=\sum_{\nu \geq 0}N^{-\nu}D^{\alpha}D_{j}L_{\nu}(x,\partial)f. 
\end{split}
\]
Thus, by induction, the asymptotic expansion \eqref{expD2} holds for any $\alpha$ uniformly on $P$. 
Since $A(x)$ is invertible on $P^{o}$, we conclude the assertion.  \hfill $\square$
\vspace{10pt}

\begin{remark}
\label{remark10}
{\rm One may think that the theorem can be proved by using 
the method of stationary phase. Indeed, we have the formula
\begin{equation}
\label{intrep}
B_{N}(f)(x)=\frac{1}{(2\pi)^{m}}\int_{\mathbb{R}^{m}} \widehat{f}(\xi)\varphi_{N,x}(\xi)\,d\xi
=\left(
\frac{N}{2\pi}
\right)^{m} \int_{\mathbb{R}^{2m}} e^{-N\Phi(x,y,\xi)}f(y)\,dyd\xi, 
\end{equation}
where $\widehat {f}$ is the Fourier transform of an extension of $f \in C^{\infty}(P)$ to 
the whole space $\mathbb{R}^{m}$ as a compactly supported smooth function which 
is also denoted by $f$, and the function $\varphi_{N,x}$ is the characteristic function, 
\[
\varphi_{N,x}(\xi)=\int_{P}e^{iz\xi}\,d\mathcal{B}_{x}^{N}(z), 
\]
of the probability measure $d\mathcal{B}_{x}^{N}$. Since $d\mathcal{B}_{x}^{N}$ is 
defined by the $N$-th convolution power of the Bernstein measure $\mathcal{B}$, we have 
\[
\varphi_{N,x}(\xi)=\varphi (x,\xi/N)^{N},\quad \varphi (x,\xi):=\int_{P}e^{iz\xi}\,d\mathcal{B}_{x}(z). 
\]
Then, the function $\Phi (x,y,\xi)$ in the right hand side of \eqref{intrep} is given by 
\[
\Phi(x,y,\xi)=i\ispa{y,\xi}-\log \varphi (x,\xi). 
\]
Thus, one would find asymptotic expansion for $B_{N}(f)$ by using the method 
of stationary phase (for example, using Theorem 7.7.5 in \cite{Ho1}) with the phase function $\Phi (x,y,\xi)$. 
Note here that, in this computation using stationary phase method, 
one might not need to assume the condition (2) in Definition \ref{Bmeasure} for
the section $\mathcal{B}:P \to \mathcal{M}(P)$. 
Theorem 7.7.5 in \cite{Ho1} gives us 
an effective formula for each term of the expansion. However, it is not still quite easy to compute 
each term of the expansion. 
Instead of using the method of stationary phase, we used in the above proof a beautiful 
idea given by H\"{o}rmander (\cite{Ho2}) which gives a computable 
representation \eqref{polyN}, \eqref{polyn2}, \eqref{diffop} of each 
differential operator $L_{\nu}(x,\partial)$ when $\mathcal{B}$ is a Bernstein measure. }
\end{remark}

\subsection{Proof of Theorem \ref{char}}
\label{proof1}

We keep the notation described before the statement of Theorem \ref{char} in Section \ref{Bernstein}. 
First, we note that the functions $m_{S,c,\alpha}$, $\alpha \in S$ defined in \eqref{Scfunct} are 
continuous up to the boundary $\partial P$ of $P$. 
Before giving a proof of this fact, we describe the values of $m_{S,c,\alpha}$ at $x \in \partial P$. 
Let $K$ be a (relatively open) face of $P$. 
Let 
\[
X_{K}={\rm span}_{\mathbb{R}}\{\alpha -\beta\,;\,\alpha,\beta \in S \cap \overline{K}\}, 
\]
and let $X_{K}^{\perp} \subset \mathbb{R}^{m}$ be the annihilator of $X_{K}$. 
Then, define the map $\mu_{K}:X \to K$ by 
\begin{equation}
\label{momentF}
\mu_{K}(\tau):=\sum_{\alpha \in S \cap \overline{K}}
\frac{c(\alpha)e^{\ispa{\alpha,\tau}}}
{\sum_{\beta \in S \cap \overline{K}}c(\beta)e^{\ispa{\beta,\tau}}}\alpha,\quad 
\tau \in X. 
\end{equation}
If $\tau,\tau' \in X$ satisfy $\tau-\tau' \in X_{K}^{\perp}$, then $\mu_{K}(\tau)=\mu_{K}(\tau')$, 
and hence the above defines a map $\mu_{K}:X_{K} \cong \mathbb{R}^{m}/X_{K}^{\perp} \to K$, 
and it is a diffeomorphism from $X_{K}$ onto $K$ (\cite{Fu}). 
For any $u \in \mathbb{R}^{m}$, we define $\lambda (u)=\min_{y \in P}\ispa{y,u}$. 
We fix $u \in X_{K}^{\perp}$ satisfying
\[
\overline{K}=\{y \in P\,;\,\ispa{y,u}=\lambda(u)\}. 
\]
Then, for any $x=\mu_{K}(\tau) \in K$ with $\tau \in \mathbb{R}^{m}$, 
we have $\lim_{t \to +\infty}\mu_{S,c}(\tau-tu)=x$ and 
\begin{equation}
\label{boundaryV}
\begin{split}
m_{S,c,\alpha}(x)&=\lim_{t \to +\infty}m_{S,c,\alpha}(\mu_{S,c}(\tau-tu)) \\
&=
\left\{
\begin{array}{ll}
0 & \alpha \not \in S \cap \overline{K},\\
\frac{c(\alpha)e^{\ispa{\alpha,\tau}}}
{\sum_{\beta \in S \cap \overline{K}}c(\beta)e^{\ispa{\beta,\tau}}} & \alpha \in S \cap \overline{K}.   
\end{array}
\right.
\end{split}
\end{equation}

\begin{lemma}
\label{contimS}
The functions $m_{S,c,\alpha}(x)$ $(\alpha \in S)$ are continuous on $P$. 
\end{lemma}

\begin{proof}
Let $K \subset \partial P$ be a relatively open face of $P$. Let $x_{n} \in P^{o}$ and $x \in K$ satisfy $x_{n} \to x$ as $n \to \infty$. 
We show that $m_{S,c,\alpha}(x_{n}) \to m_{S,c,\alpha}(x)$.
(For the case where $x_{n}$ is contained in a face $L \neq P^{o}$, one can discuss as in the following 
with replacing $P^{o}$ by $L$.) Let $\tau_{n}:=\tau_{S,c}(x_{n}) \in \mathbb{R}^{m}$. 
Take $u \in X_{K}^{\perp}$ such that $\overline{K}=\{y \in P\,;\,\ispa{y,u}=\lambda (u)\}$ with $\lambda (u)=\min_{y \in P}\ispa{y,u}$. 
For any $A \subset S$, we define $\chi_{A}(\tau)=\sum_{\alpha \in A}c(\alpha)e^{\ispa{\alpha,\tau}}$. 
We set $S_{K}:=S \cap \overline{K}$ and $c_{K}=\min_{\alpha \in S \setminus S_{K}} \{\ispa{\alpha,u}-\lambda (u)\}$. 
We note that $\ispa{\alpha,u}-\lambda(u)=0$ if and only if $\alpha \in S_{K}$, 
and hence $c_{K}>0$. Since $0 < \ispa{x_{n},u}-\lambda (u) \to 0$ as $n \to \infty$, we have 
\[
0 < c_{K}\frac{\chi_{S \setminus S_{K}}(\tau_{n})}
{\chi_{S_{K}}(\tau_{n})+\chi_{S \setminus S_{K}}(\tau_{n})} 
\leq \ispa{\mu_{S,c}(\tau_{n}),u}-\lambda (u) \to 0, 
\]
and which shows
\begin{equation}
\label{auxCo1}
\frac{\chi_{S \setminus S_{K}}(\tau_{n})}
{\chi_{S_{K}}(\tau_{n})+\chi_{S \setminus S_{K}}(\tau_{n})}  \to 0, \quad 
\frac{\chi_{S \setminus S_{K}}(\tau_{n})}{\chi_{S_{K}}(\tau_{n})} \to 0\quad (n \to \infty).
\end{equation}
From this, we have $\lim_{n \to \infty}m_{S,c,\alpha}(x_{n})=0=m_{S,c,\alpha}(x)$ when $\alpha \in S \setminus S_{K}$. 
Decompose $\tau_{n} \in \mathbb{R}^{m}$ according to the decomposition $\mathbb{R}^{m}=X_{K} \oplus X_{K}^{\perp}$ as $\tau_{n}=\xi_{n}+u_{n}$, 
$\xi_{n} \in X_{K}$, $u_{n} \in X_{K}^{\perp}$. 
Then, \eqref{auxCo1} also shows 
\[
\mu_{K}(\xi_{n})=\mu_{K}(\tau_{n})=
\sum_{\alpha \in S_{K}}
\frac{c(\alpha)e^{\ispa{\alpha,\tau_{n}}}}
{\chi_{S_{K}}(\tau_{n})}\alpha
\to x \quad (n \to \infty).  
\]
Since $\mu_{K}:X_{K} \to K$ is a diffeomorphism, we take $\xi \in X_{K}$ such that $\mu_{K}(\xi)=x$. 
Then, the above means $\lim_{n \to \infty}\mu_{K}(\xi_{n})=\mu_{K}(\xi)$, and hence we have $\lim_{n \to \infty}\xi_{n}=\xi$. 
From this we have, for $\alpha \in S_{K}$,  
\[
m_{S,c,\alpha}(x_{n})=
\frac{c(\alpha)e^{\ispa{\alpha,\xi_{n}}}}
{\chi_{S_{K}}(\xi_{n})}(1+o(1)) \to 
\frac{c(\alpha)e^{\ispa{\alpha,\xi}}}
{\chi_{S_{K}}(\xi)}=m_{S,c,\alpha}(x)
\]
as $n \to \infty$, which shows the assertion. 
\end{proof}

\begin{remark}
\label{remarkFIX}
{\rm For the original Bernstein polynomials $B_{N}(f)$ defined in \eqref{original}, one has $B_{N}(f)(0)=f(0)$, $B_{N}(f)(1)=f(1)$. 
Our Bernstein approximations $B_{N}(f)$ defined by a finitely supported Bernstein measure $\mathcal{B}_{S,c}$ also 
have similar property. In fact, by \eqref{boundaryV}, if $x$ is a vertex of $P$, 
one has $m_{S,c,x}(x)=1$, $m_{S,c,\alpha}(x)=0$ when $\alpha \neq x$. 
This combined with \eqref{coeff} shows that $B_{N}(f)(x)=f(x)$ when $x$ is a vertex of $P$. }
\end{remark}

Now, we define the function $\delta_{S,c} \in C^{\infty}(P^{o})$ by 
\begin{equation}
\label{LegendreD}
\delta_{S,c}(x)=
\log\left(
\sum_{\alpha \in S}c(\alpha)e^{\ispa{\alpha,\tau_{S,c}(x)}}
\right)
-\ispa{x,\tau_{S,c}(x)}, 
\end{equation}
where, as in Section \ref{Bernstein}, 
$\tau_{S,c}:P^{o} \to \mathbb{R}^{m}$ is the inverse map of the diffeomorphism $\mu_{S,c}:\mathbb{R}^{m} \to P^{o}$. 
Then, we have 
\begin{equation}
\label{repofmSc}
m_{S,c,\alpha}(x)=c(\alpha)e^{-\delta_{S,c}(x)+\ispa{\alpha-x,\tau_{S,c}(x)}}. 
\end{equation}

\vspace{10pt}

{\sc Completion of proof of Theorem \ref{char}.}\hspace{5pt}
As in \cite{TZ}, we have $\nabla \delta_{S,c}(x)=-\tau_{S,c}(x)$ 
and $\nabla^{2}\delta_{S,c}(x)=-A_{S,c}(x)^{-1}$ for $x \in P^{o}$,  
where $A_{S,c}$ is defined by 
\begin{equation}
\label{ASc}
A_{S,c}(x)=\sum_{\alpha \in S}m_{S,c,\alpha}(x)(\alpha -x)\otimes (\alpha -x), 
\end{equation}
which is non-degenerate on $P^{o}$ (see \cite{TZ}). 
Then, it is not hard to show directly that the measure \eqref{fB} defines a Bernstein measure on $P$ 
with the defining matrix $A_{S,c}(x)^{-1}$. 

Next, let us prove the converse. Let $\mathcal{B}:P \to \mathcal{M}(P)$ be a finitely supported Bernstein measure, 
and let $S ={\rm supp}(\mathcal{B}(x))$ for some, and any $x \in P^{o}$. Let $K:P^{o} \to {\rm Sym}(m,\mathbb{R})$ 
be the defining matrix of $\mathcal{B}$. 
We write $\mathcal{B}(x)=\sum_{\alpha \in S}m_{\alpha}(x)\delta_{\alpha}$, $x \in P$. 
Then, by Lemma \ref{InverseMatrix}, the matrix $K(x)$ is non-degenerate. 
Now, fix an arbitrary $b^{*} \in P^{o}$ and $V>0$.  
Define the function $\delta_{b^{*}} \in C^{\infty}(P^{o})$ and 
the map $\tau_{b^{*}}:P^{o} \to \mathbb{R}^{m}$ by 
\[
\begin{gathered}
\tau_{b^{*}}(x)=\int_{0}^{1}K(b^{*}+s(x-b^{*}))(x-b^{*})\,ds,\\
\delta_{b^{*}}(x)=\log V -\int_{0}^{1}
\ispa{\tau_{b^{*}}(b^{*}+s(x-b^{*})),x-b^{*}}\,ds. 
\end{gathered}
\]
Lemma \ref{commute2} shows that $(d\tau_{b^{*}})_{x}=K(x)$, $\nabla \delta_{b^{*}}(x)=-\tau_{b^{*}}(x)$, $x \in P^{o}$. 
Since $m_{\alpha}(x)>0$ on $P^{o}$, we define $g_{\alpha}(x)=\log m_{\alpha}(x) -\ispa{\alpha-x,\tau_{b^{*}(x)}}$. 
Then, since $\nabla g_{\alpha}=\tau_{b^{*}}$, there exists a constant $c(\alpha)>0$ such that 
$g_{\alpha} =-\delta_{b^{*}}+\log c(\alpha)$ for each $\alpha \in S$ and hence 
\begin{equation}
\label{repofm}
m_{\alpha}(x)=c(\alpha)e^{-\delta_{b^{*}}(x)+\ispa{\alpha-x,\tau_{b^{*}}(x)}}. 
\end{equation} 
Taking the sum of \eqref{repofm} over $\alpha \in S$, we have 
\[
\delta_{b^{*}}(x)=
\log \left(
\sum_{\alpha \in S}c(\alpha)e^{\ispa{\alpha,\tau_{b^{*}}(x)}}
\right)
-\ispa{x,\tau_{b^{*}}(x)}.
\]
Since $\tau_{b^{*}}(b^{*})=0$, 
we have $m_{\alpha}(b^{*})=c(\alpha)/V$.  
It follows from this and the identity $\sum_{\alpha \in S}m_{\alpha}(b^{*})\alpha=b^{*}$ that 
\[
V=\sum_{\alpha \in S}c(\alpha),\quad 
b^{*}=\frac{1}{V}\sum_{\alpha \in S}c(\alpha)\alpha. 
\]
From these formulas, we obtain $\delta_{b^{*}}(b^{*})=\log V$. 
It is not hard to show that the map $\tau_{b^{*}}$ is injective and hence is a diffeomorphism 
from $P^{o}$ to the image of $\tau_{b^{*}}$. (Note that the differential of $\tau_{b^{*}}$ is 
$K$, which is positive definite on $P^{o}$.) 
Denote its inverse by $\mu_{b^{*}}:{\rm Im}(\tau_{b^{*}}) \to P^{o}$. 
For any $x \in P^{o}$ and $\tau \in {\rm Im}(\tau_{b^{*}})$, we define $f_{\tau}(x)=\ispa{x,\tau}+\delta_{b^{*}}(x)$. 
Then, we have $\nabla f_{\tau}(x)=\tau -\tau_{b^{*}}(x)$ and $\nabla^{2}f_{\tau}(x)=-K(x)$. 
Thus, the point $x=\mu_{b^{*}}(\tau)$ is a unique critical point of $f_{\tau}$,  
and $f_{\tau}$ attains its maximum there. Then, we obtain 
\[
f_{\tau}(x) \leq f_{\tau}(\mu_{b^{*}}(\tau))=
\log \chi_{S,c}(\tau),
\]
where the function $\chi_{S,c}$, defined by 
\begin{equation}
\label{potential}
\chi_{S,c}(\tau)=
\sum_{\alpha \in S}c(\alpha)e^{\ispa{\alpha,\tau}}, 
\end{equation}
depends only on the constants $c(\alpha)>0$ and the set $S$. 
In the above inequality, the equality holds if and only if $x=\mu_{b^{*}}(\tau)$. 
Therefore, we obtain 
\[
\delta_{b^{*}}(x) \leq \log \chi_{S,c}(\tau) -\ispa{x,\tau}
\]
for every $\tau \in {\rm Im} (\tau_{b^{*}})$ and equality holds if and only if $\tau =\tau_{b^{*}}(x)$. 
Then, as in \cite{TZ}, we obtain 
\[
\delta_{S,c}(x) =\min_{\tau \in \mathbb{R}^{m}}
\left(
\log \chi_{S,c}(\tau)-\ispa{x,\tau}
\right) \leq \delta_{b^{*}}(x). 
\]
Since ${\rm Im}(\tau_{b^{*}}) \subset \mathbb{R}^{m}$ is open, the point $\tau=\tau_{b^{*}}(x)$ is a local minimum 
of the function $\tau \mapsto \log \chi_{S,c}(\tau)-\ispa{x,\tau}$ on $\mathbb{R}^{m}$. 
Since, the local minimum of this function is unique, and it is given by $\tau=\tau_{S,c}(x)$, 
we conclude $\tau_{b^{*}}(x)=\tau_{S,c}(x)$ for every $x \in P^{o}$. This completes the proof. \hfill $\square$

\begin{remark}
\label{remark33}
{\rm In the proof of the converse direction of Theorem \ref{char} above, there are no restriction for 
the weight function $c:S \to \mathbb{R}_{>0}$, which is defined as an integral constant. 
This is because any choice of the weight function $c:S \to \mathbb{R}_{>0}$ can define a Bernstein measure. 
However, two different weight functions might give the same Bernstein measure. 
In fact, when $P$ is a standard simplex and $S=P \cap \mathbb{Z}^{m}$, the set of vertices of the simplex $P$, 
any choice of the weight function produce the Bernstein measure given in Example \ref{example1}.  
In general, let $A \subset P$ be a finite set whose convex hull is $P$. 
Let $\Delta_{A} \subset \mathcal{M}(P)$ denote the convex hull of the Dirac measures $\delta_{\alpha}$ with $\alpha \in A$. 
Then, the image of the finitely supported Bernstein measures $\mathcal{B}:P \to \mathcal{M}(P)$ with the support $S$ 
are contained in $\Delta_{S}$, and the image of the dilated convolution powers $d\mathcal{B}_{x}^{N}$ 
are contained in $\Delta_{\frac{1}{N}S_{N}}$. 
When $P$ is a standard $m$-dimensional simplex and $S=P \cap \mathbb{Z}^{m}$, $\Delta_{S}$ is also a simplex with 
the same dimension, and the restriction of the barycenter map to $\Delta_{S}$ is a diffeomorphism between $\Delta_{S}$ and $P$. 
Hence, in this case, there are only one section of the barycenter map whose support (in the sense of Definition \ref{defSupport}) is $S$. }
\end{remark}

\subsection{Large deviations principle}
\label{LargeDP}

In this subsection, we discuss the large deviations principle 
for finitely supported Bernstein measures. 
Throughout this subsection, let $\mathcal{B}=\mathcal{B}_{S,c}:P \to \mathcal{M}(P)$ denote 
the finitely supported Bernstein measure with the support $S \subset P$ and the weight $c:S \to \mathbb{R}_{>0}$. 
For each $\alpha \in S$, the coefficient of $\mathcal{B}$ is denoted by $m_{S,c,\alpha}(x)$ defined in \eqref{Scfunct}.

For any $x \in P$, we define a function $\chi_{x}=\chi_{S,c,x}$ by 
\begin{equation}
\label{potentialX}
\chi_{x}(\tau):=\sum_{\alpha \in S}m_{S,c,\alpha}(x)e^{\ispa{\tau,\alpha}}
=\int_{P}e^{\ispa{\tau,z}}\,d\mathcal{B}_{x}(z),\quad \tau \in \mathbb{R}^{m}. 
\end{equation}
By using the function $\chi_{S,c}$ defined in \eqref{potential}, we find that 
\begin{equation}
\label{potentialXF}
\chi_{x}(\tau)=\frac{\chi_{S,c}(\tau_{S,c}(x)+\tau)}{\chi_{S,c}(\tau_{S,c}(x))}. 
\end{equation}

\begin{lemma}
We fix $x \in P$. Then the dilated convolution powers $d\mathcal{B}_{x}^{N}$ defined in \eqref{convolution} 
from the finitely supported Bernstein measure $\mathcal{B}=\mathcal{B}_{S,c}$ 
satisfy the large deviations principle with the speed $N$ and the good rate function given by 
\begin{equation}
\label{rateF}
I^{x}(y)=\sup_{\tau \in \mathbb{R}^{m}}
\{\ispa{y,\tau}-\Lambda^{x} (\tau)\},\quad 
\Lambda^{x}(\tau)=\log \chi_{x}(\tau). 
\end{equation}
\end{lemma}

\begin{proof}
Consider the infinite product of $d\mathcal{B}_{x}$ 
on the infinite product space $\Omega =P \times \cdots \times  P \times \cdots$. 
Let $X_{j,x}:\Omega \to P$ denote the projection in the $j$-th coordinate. 
Then $X_{j,x}$'s form a sequence of independent identically distributed 
random vectors with the distribution $d\mathcal{B}_{x}$. 
Then, the distribution of the empirical means $S_{N,x}:=\frac{1}{N}\sum_{j=1}^{N}X_{j,x}$ is given 
by the dilated convolution powers $d\mathcal{B}_{x}^{N}$. Thus the assertion follows from Cram\'{e}r's theorem (\cite{DZ}). 
\end{proof}

Next, we describe the rate function \eqref{rateF} by using \eqref{momentF}, \eqref{boundaryV}. 
To describe the rate function \eqref{rateF} more concretely, we need the fact that 
the function $\delta_{S,c}$ in \eqref{LegendreD} is continuous on $P$. 
Let $K$ be a (relatively open) face of the polytope $P$. 
Let $\tau_{K}:K \to X_{K}$ be the inverse of the map $\mu_{K}:X_{K} \to K$ defined in \eqref{momentF}. 
When $K=P^{o}$, we have $\mu_{P^{o}}=\mu_{S,c}$ and $\tau_{P^{o}}=\tau_{S,c}$. 
As in the proof of Lemma \ref{contimS}, for each $A \subset S$, we set $\chi_{A}(\tau)=\sum_{\alpha \in A}c(\alpha)e^{\ispa{\alpha,\tau}}$ 
and $S_{K}=S \cap \overline{K}$. 
Define a function $\delta_{K}$ on $K$ by 
\begin{equation}
\label{LegendreDK}
\delta_{K}(x)=\log \chi_{S_{K}}(\tau_{K}(x))-\ispa{x,\tau_{K}(x)},\quad x \in K. 
\end{equation}
Note that, we have $\delta_{P^{o}}=\delta_{S,c}$. 

\begin{lemma}
\label{contiD}
The function $\delta_{S,c}$ is continuous on $P$, and its restriction to each face $K$ is given by $\delta_{K}$. 
\end{lemma}

\begin{proof}
As in the proof of Lemma \ref{contimS}, take $u \in X_{K}^{\perp}$ such that $\overline{K}=\{y \in P\,;\,\ispa{y,u}=\lambda(u)\}$ 
with $\lambda (u)=\min_{z \in P}\ispa{z,u}$. 
Then, first of all, we claim that, for $x \in K$, the following holds: 
\begin{equation}
\label{limit1}
\lim_{t \to +\infty}\delta_{S,c}(\mu_{S,c}(\tau_{K}(x)-tu))=\delta_{K}(x). 
\end{equation}
Let us prove \eqref{limit1}. It is easy to show that, for any $x \in K$, 
\begin{equation}
\label{momentEXP}
|\ispa{\mu_{S,c}(\tau_{K}(x)-tu)-x},u| \leq Ce^{-tc_{K}}\frac{\chi_{S \setminus S_{K}}(\tau_{K}(x))}{\chi_{S_{K}}(\tau_{K}(x))}, 
\end{equation}
where we set $c_{K}:=\min_{\alpha \in S \setminus S_{K}}\ispa{\alpha,u}-\lambda(u)$. 
For any $\xi \in X_{K}$, we set 
\begin{equation}
\label{errorEXP}
0<R_{K}(t,\xi):=\sum_{\alpha \in S \setminus S_{K}}c(\alpha)e^{\ispa{\alpha,\xi}-t(\ispa{\alpha,u}-\lambda(u))} 
\leq e^{-tc_{K}}\chi_{S \setminus S_{K}}(\xi). 
\end{equation}
Since $x \in K$, we have $\ispa{x,u}=\lambda(u)$, and hence 
\begin{equation}
\label{auxCONTI}
\begin{split}
\delta_{S,c}&(\mu_{S,c}(\tau_{K}(x)-tu)) \\
=&\log \chi_{S_{K}}(\tau_{K}(x))-\ispa{\mu_{S,c}(\tau_{K}(x)-tu),\tau_{K}(x)}+\\
&+t\ispa{\mu_{S,c}(\tau_{K}(x)-tu)-x,u}+
\log(1+R_{K}(t,\tau_{K}(x))/\chi_{S_{K}}(\tau_{K}(x))).  
\end{split}
\end{equation}
Now, by \eqref{momentEXP}, \eqref{errorEXP} and the fact that $\mu_{S,c}(\tau_{K}(x)-tu)$ tends to $x \in K$ as $t \to +\infty$, 
the right hand side of \eqref{auxCONTI} converges to $\delta_{K}(x)$, which shows \eqref{limit1}. 

Next, we take $x \in K$, and $\{x_{n}\} \subset P^{o}$ such that $x_{n} \to x$ as $n \to \infty$. 
Let $\xi=\tau_{K}(x)$ and $\tau_{n}=\xi_{n}+u_{n}=\tau_{S,c}(x_{n})$ with $\xi_{n} \in X_{K}$, $u_{n} \in X_{K}^{\perp}$. 
In the proof of Lemma \ref{contimS}, we have proved that $\xi_{n} \to \xi$. Let $p_{K}(n)=\ispa{z,u_{n}}$ with $z \in \overline{K}$, which does 
not depend on the choice of $z \in \overline{K}$. 
We note that $\chi_{S_{K}}(\xi_{n}+u_{n})=e^{p_{K}(n)}\chi_{S_{K}}(\xi_{n})$. 
From this and \eqref{auxCo1}, for each $\alpha \in S \setminus S_{K}$, we have 
\[
c(\alpha)e^{\ispa{\alpha,\xi_{n}}+\ispa{\alpha,u_{n}}-p_{K}(n)} \leq e^{-p_{K}(n)}\chi_{S\setminus S_{K}}(\xi_{n}+u_{n})\to 0,  
\]
and hence $\ispa{\alpha,u_{n}}-p_{K}(n) \to -\infty$ as $n \to \infty$ for each $\alpha \in S \setminus S_{K}$. 
Now, by using this fact, we have 
\[
\begin{split}
&-\ispa{x_{n},u_{n}}+\log \chi_{S,c}(\tau_{n}) \\
=&\log \chi_{S_{K}}(\xi_{n})-\sum_{\alpha \in S \setminus S_{K}}
\frac{c(\alpha)e^{\ispa{\alpha,\tau_{n}}}}{\chi_{S,c}(\tau_{n})}(\ispa{\alpha,u_{n}}-p_{K}(n))+o(1)
\to \log \chi_{S_{K}}(\xi). 
\end{split}
\]
In the above, one can compute the term $\ispa{x_{n},u_{n}}$ by using the relation $x_{n}=\mu_{S,c}(\xi_{n}+u_{n})$ and 
the definition \eqref{momentM} of the map $\mu_{S,c}$. 
Since $\delta_{S,c}(x_{n})=-\ispa{x_{n},\xi_{n}+u_{n}}+\log \chi_{S,c}(\tau_{n})$, we conclude that $\delta_{S,c}(x_{n}) \to \delta_{K}(x)$. 
In case where $\{x_{n}\}$ is contained in a face $L \neq P^{o}$, one can 
discuss in the same way as above with replacing $P^{o}$ by $L$ to conclude the assertion. 
\end{proof}

\begin{remark}
\label{deltaBV}
{\rm The function $\delta_{K}$ on the face $K$ defined in \eqref{LegendreDK} is continuous on $\overline{K}$. 
This can be shown in the same way as in the above proof. The restriction of $\delta_{K}$ on a (relatively open) face $L$ of $K$ 
is given by $\delta_{L}$. }
\end{remark}

\begin{proposition}
\label{LDP}
Let $K$ be a relatively open face of $P$. Let $x \in K$. 
Then, the rate function $I^{x}(y)$ is given by the following: 
\begin{equation}
\label{ConcreteRF}
I^{x}(y)=
\left\{
\begin{array}{ll}
+\infty & y \not \in \overline{K},\\
\delta_{K}(x)-\delta_{K}(y)+\ispa{x-y,\tau_{K}(x)} & y \in \overline{K}.  
\end{array}
\right.
\end{equation}
\end{proposition}

\begin{proof}
Let $K$ be a face, and we fix $x \in K$. 
By \eqref{boundaryV}, \eqref{potentialX}, 
we have $\chi_{x}(\tau)=\frac{\chi_{S_{K}}(\tau_{K}(x)+\tau)}{\chi_{S_{K}}(\tau_{K}(x))}$. 
Thus, we have 
\begin{equation}
\label{generalF}
\begin{gathered}
I^{x}(y)=J_{K}(y)+\delta_{K}(x)+\ispa{x-y,\tau_{K}(x)},\\
J_{K}(y)=\sup_{\tau \in \mathbb{R}^{m}}\{\ispa{y,\tau}-\log \chi_{S_{K}}(\tau)\}. 
\end{gathered}
\end{equation}
This holds for any $y \in\mathbb{R}^{m}$. 
For any $y \in \mathbb{R}^{m}$, we set 
\[
I_{K}(y)=\sup_{\xi \in X_{K}}\{\ispa{y,\xi}-\log \chi_{S_{K}}(\xi)\}. 
\]
In \eqref{generalF}, we decompose $\tau=\xi +u$ with $\xi \in X_{K}$, $u \in X_{K}^{\perp}$. 
Then, for $y \in \overline{K}$, we have $\ispa{y,\xi+u}-\log \chi_{x}(\xi+u)=\ispa{y,\xi}-\log \chi_{x}(\xi)$ and hence $I_{K}(y)=J_{K}(y)$ 
for each $y \in \overline{K}$. 
Now, first of all, we assume $y \in K$. Then, for $\xi \in X_{K}$, we have $\nabla_{\xi} \log \chi_{S_{K}}(\xi)=\mu_{K}(\xi)$, 
which shows that $\xi=\tau_{K}(y)$ is a unique critical point of the function $\xi \mapsto \ispa{y,\xi}-\log \chi_{S_{K}}(\xi)$. 
The Hessian of the function $\log \chi_{S_{K}}(\xi)$ is non-negative, and hence we have 
\begin{equation}
\label{indepX}
I_{K}(y)=\ispa{y,\tau_{K}(y)}-\log \chi_{S_{K}}(\tau_{K}(y))=-\delta_{K}(y), 
\end{equation}
which shows \eqref{ConcreteRF} for $y \in K$. 
We set $F_{K}^{x}(y)=\delta_{K}(x)-\delta_{K}(y)+\ispa{x-y,\tau_{K}(x)}$ for $y \in \overline{K}$, 
which is continuous in $y \in \overline{K}$. Since $I^{x}(y)$ is lower semi-continuous in $y$, 
we have $F_{K}^{x}(y) \geq I^{x}(y)$ for any $y \in \overline{K}$. 
Now, assume that $y$ is contained in a face $L$ of $\overline{K}$, 
and we show the inequality $F_{K}^{x}(y) \leq I^{x}(y)$ for such a point $y$, which implies \eqref{ConcreteRF}.  
Since $I_{K}(y)=J_{K}(y)$ for any $y \in \overline{K}$, it suffices to show that $I_{K}(y) \geq -\delta_{K}(y)$. 
Take $u \in X_{L}^{\perp}$ such that $\overline{L}=\{z \in P\,;\,\ispa{z,u}=\lambda (u)\}$ with 
$\lambda (u)=\min_{z \in P}\ispa{z,u}$. Then, for any $\xi \in X_{L}$ and $t>0$, we have 
\[
\begin{gathered}
\chi_{S_{K}}(\xi-tu)=e^{-t\lambda(u)}(\chi_{S_{L}}(\xi)+R_{L}(t,\xi,u)), \\
R_{L}(t,\xi,u)=\sum_{\alpha \in S_{K} \setminus S_{L}}
c(\alpha)e^{\ispa{\alpha,\xi}-t(\ispa{\alpha,u}-\lambda (u))}. 
\end{gathered}
\]
Note that $R_{K}(t,\xi,u) \to 0$ as $t \to +\infty$. 
Since $y \in L$, we have $\lambda (u)=\ispa{y,u}$, and hence 
\[
\log \chi_{S_{K}}(\xi-tu) =-t\ispa{y,u}+\log \chi_{S_{L}}(\xi)+\log (1+o_{t}(1)), 
\]
where $o_{t}(1)$ denotes a term tending to zero as $t \to +\infty$. Since $I_{K}(y)=J_{K}(y) \geq \ispa{y,\xi-tu}-\log \chi_{S_{K}}(\xi-tu)$ 
for any $\xi \in X_{L}$ and $t>0$, 
we obtain 
\[
I_{K}(y) \geq \ispa{y,\xi}-\log \chi_{S_{L}}(\xi),\quad \xi \in X_{L}. 
\]
Since the supremum over $\xi \in X_{L}$ of the right hand side above is $-\delta_{L}(y)=-\delta_{K}(y)$, 
we conclude $I_{K}(y) \geq -\delta_{K}(y)$ for $y \in L$. 
Finally, we show that $I^{x}(y)=+\infty$ when $y \not \in \overline{K}$. 
We have 
\[
\ispa{y,\tau}-\log \chi_{S_{K}}(\tau)=-\log\chi_{S_{K}}(0)+\int_{0}^{1}\ispa{y-\mu_{K}(t\tau),\tau}\,dt. 
\]
Note, in the above, that $\mu_{K}(t\tau) \in K$ for any $t \in [0,1]$ and $\tau \in \mathbb{R}^{m}$. 
Since $y \not \in \overline{K}$, we can take $u \in \mathbb{R}^{m}$ such that $\ispa{y-z,u}>0$ for any $z \in \overline{K}$. 
Then, by setting $\tau=ru$ with $r >0$ in the above, we conclude that $J_{K}(y)=+\infty$ and hence $I^{x}(y)=+\infty$. 
\end{proof}

\begin{remark}
\label{remarkLDP}
{\rm When the polytope $P$ is a lattice polytope satisfying Delzant condition, 
the formula for the rate function \eqref{ConcreteRF} 
coincides with the formula given in \cite{SoZ2}, Proposition 5.2 
for the rate function in a large deviations principle of the Bergman-Bernstein measure (see Section \ref{ToricGeometry}) 
defined by the Fubini-Study metric on the toric manifold obtained through a monomial embedding. 
Indeed, for example, when $x \in P^{o}$, the rate function $I^{x}(y)$ is given by 
\[
\begin{split}
I^{x}(y)&=\delta_{S,c}(x)-\delta_{S,c}(y)+\ispa{x-y,\tau_{S,c}(x)}\\
&=-\delta_{S,c}(y)+\log \chi_{S,c}(\tau_{S,c}(x))-\ispa{y,\tau_{S,c}(x)}. 
\end{split}
\]
Then, the function $\log \chi_{S,c}$ is a K\"{a}hler potential for the Fubini-Study metric and $-\delta_{S,c}$ 
is its Legendre dual. See \cite{Z}, \cite{SoZ2} and the following section for details. }
\end{remark}

\section{Bergman-Bernstein approximations}
\label{ToricGeometry}

In \cite{Z}, Zelditch introduced the notion of the Bergman-Bernstein approximations for 
functions on Delzant polytopes $P$ (a lattice polytope with the property that 
each vertex of $P$ has exactly $m=\dim P$ edges and $m$ lattice vectors incident from the 
vertex along the edges form a $\mathbb{Z}$-basis of the lattice). 
We explain here this notion for projective toric manifolds obtained by monomial embeddings 
and the difference between Bergman-Bernstein approximations and Bernstein measures defined in this paper. 

Throughout this section, we assume that the polytope $P$ is Delzant with vertices in $\mathbb{Z}^{m}$. 
We set $S=P \cap \mathbb{Z}^{m}$ and $S_{N}=NP \cap \mathbb{Z}^{m}$. 
We fix a function $c:S \to \mathbb{R}_{>0}$. 
Let $\Phi_{S,c}:(\mathbb{C}^{*})^{m} \to \mathbb{C}P^{|S|-1}$ be the monomial embedding 
defined by 
\[
\Phi_{S,c}(z)=[c(\alpha)^{1/2}z^{\alpha}]_{\alpha \in S},\quad z \in (\mathbb{C}^{*})^{m}. 
\]
Then, the toric variety $M_{S,c}$ is defined by the Zariski closure of the image of the 
monomial embedding $\Phi_{S,c}:(\mathbb{C}^{*})^{m} \to \mathbb{C}P^{|S|-1}$ (\cite{GKZ}). 
Note that, in general, $M_{S,c}$ may have singularities. 
However, it is well-known (\cite{GKZ}) that, if $P$ satisfies the Delzant condition, then $M_{S,c}$ is 
a non-singular compact K\"{a}hler manifold. 
Consider the (restriction of the) Fubini-Study K\"{a}hler form $\omega_{{\rm FS}}$ on $M_{S,c}$. 
Then, the action of the real torus $T^{m} \subset (\mathbb{C}^{*})^{m}$ is Hamiltonian with respect 
to the symplectic form $\omega_{{\rm FS}}$. 
On the open orbit $(\mathbb{C}^{*})^{m} \cong \Phi_{S,c}((\mathbb{C}^{*})^{m})$, 
one can take a $T^{m}$-invariant K\"{a}hler potential $\varphi$. 
Since $\varphi$ is a function on $(\mathbb{C}^{*})^{m}$ invariant under $T^{m}$-action, 
it defines a function on $\mathbb{R}^{m}$, which we denote by $\varphi_{S,c}$. 
Let $\mu:M_{S,c} \to P$ be the moment map of the Hamiltonian action of $T^{m}$ on $(M_{S,c},\omega_{{\rm FS}})$. 
Since $\mu$ is also $T^{m}$-invariant, it defines a map $\mu_{S,c}:\mathbb{R}^{m} \to P$, 
which is a diffeomorphism onto the interior, $P^{o}$, of $P$. 

The symplectic potential associated to $\varphi$ is the Legendre dual $u_{\varphi}$ 
of the function $\varphi_{S,c}$ associated to the K\"{a}hler potential $\varphi$, which is defined by 
\[
u_{\varphi}(x)=\ispa{x,\tau_{S,c}(x)}-\varphi_{S,c} (\tau_{S,c}(x)),\quad x \in P^{o}, 
\]
where $\tau_{S,c}:P^{o} \to \mathbb{R}^{m}$ is the inverse of $\mu_{S,c}$. 
The Bergman-Bernstein approximation $\nu_{N}(f)$ (the notation $B_{h^{N}}(f)$ is used 
in \cite{Z}, where $h$ denotes a Hermitian metric on the hypersection bundle $L_{S,c}:=\mathcal{O}(1)|_{M_{S,c}}$ 
over $M_{S,c}$ whose curvature is $\omega_{{\rm FS}}$) of a function $f$ on $P$ is defined by 
\begin{equation}
\label{BBapprox}
\nu_{N}(f)(x)=\frac{1}{\Pi_{N}(z,z)}
\sum_{\gamma \in S_{N}}f(\gamma/N)
\frac{e^{N(u_{\varphi}(x)+\ispa{\gamma/N -x,\tau_{S,c}(x)})}}{Q_{h^{N}}(\gamma)}, 
\end{equation}
where $Q_{h^{N}}(\gamma)$ is the squared $L^{2}$-norm of the monomial with weight $\gamma$, which
is regarded as an element of $H^{0}(M_{S,c},L_{S,c}^{\otimes N})$, 
the function $\Pi_{N}(z,z)$ is the Bergman-Szeg\"{o} kernel for $H^{0}(M_{S,c},L_{S,c}^{\otimes N})$, 
and $z \in (\mathbb{C}^{*})^{m}$ satisfies $\mu(z)=x$. 

To compare the Bergman-Bernstein approximation \eqref{BBapprox} with our Bernstein measures, 
we take the K\"{a}hler potential 
\[
\varphi (z)=\log \sum_{\alpha \in S}c(\alpha)|z^{\alpha}|^{2},\quad z \in (\mathbb{C}^{*})^{m},  
\]
of $\omega_{{\rm FS}}$ on $(\mathbb{C}^{*})^{m}$. 
Then, the corresponding function $\varphi_{S,c}$ coincides with the function $\chi_{S,c}$ on $\mathbb{R}^{m}$ 
defined in \eqref{potential}. 
Therefore, the symplectic potential $u_{\varphi}$ coincides with the function $-\delta_{S,c}$ 
defined in \eqref{LegendreD}.  The quantity $Q_{h^{N}}(\gamma)$ is, as in \cite{Z}, 
given by 
\[
Q_{h^{N}}(\gamma)=\int e^{-N\delta_{S,c}(x)+\ispa{\gamma -Nx,\tau_{S,c}(x)}}\,dx.
\]
Note that the restriction of the moment map $\mu$ to the open orbit is given by 
\[
\mu (z)=\sum_{\alpha \in S}\frac{c(\alpha)|z^{\alpha}|^{2}}
{\sum_{\beta \in S}c(\beta)|z^{\beta}|^{2}}\alpha, \quad 
z \in (\mathbb{C}^{*})^{m}. 
\]
Thus, the map $\mu_{S,c}:\mathbb{R}^{m} \to P^{o}$ induced by the moment map $\mu$ 
is nothing but the map defined in \eqref{momentM}. 

\begin{lemma}
Let $m_{S,N}^{\gamma}(x)$ be the function on $P$ defined by \eqref{coeff} $($with $m_{\alpha}$ 
replaced by $m_{S,c,\alpha}$ defined in \eqref{Scfunct}$)$.  
Then, the Bergman-Bernstein approximation \eqref{BBapprox} is written as 
\begin{equation}
\label{BB}
\nu_{N}(f)(x)=\frac{1}{\Pi_{N}(x)}\sum_{\gamma \in S_{N}}
f(\gamma/N)\frac{m_{S,N}^{\gamma}(x)}{R_{N}(\gamma)}, 
\end{equation}
where the quantity $R_{N}(\gamma)$ and the function $\Pi_{N}(x)$ are given by 
\begin{equation}
\label{kernel}
\begin{gathered}
R_{N}(\gamma)=\int_{P} m_{S,N}^{\gamma}(x)\,dx,\\
\Pi_{N}(x)=\sum_{\gamma \in S_{N}}\frac{m_{S,N}^{\gamma}(x)}{R_{N}(\gamma)}, \quad x \in P. 
\end{gathered}
\end{equation}
\end{lemma}

\begin{proof}
First, we note that the function $m_{S,c,\alpha}(x)$ is written as \eqref{repofmSc}. 
Thus, the functions $m_{S,N}^{\gamma}(x)$ defined in \eqref{coeff} is written as 
\[
m_{S,N}^{\gamma}(x)=\mathcal{P}_{N}(\gamma)e^{-N\delta_{S,c}(x)+\ispa{\gamma-Nx,\tau_{S,c}(x)}}, 
\]
where $\mathcal{P}_{N}(\gamma)=\mathcal{P}_{S,c,N}(\gamma)$ is the weighted number of lattice path, 
\[
\mathcal{P}_{N}(\gamma):=\sum_{\beta_{1},\ldots,\beta_{N} \in S\,;\,\beta_{1}+\cdots +\beta_{N}=\gamma}c(\beta_{1})\cdots c(\beta_{N}). 
\]
From this, we have 
\begin{equation}
\label{RQP}
R_{N}(\gamma)=Q_{h^{N}}(\gamma)\mathcal{P}_{N}(\gamma), 
\end{equation} 
which shows that 
\[
\frac{e^{N(-\delta_{S,c}(x)+\ispa{\gamma/N-x,\tau_{S,c}(x)})}}
{Q_{h^{N}}(\gamma)}
=\frac{m_{N}^{\gamma}(x)}{R_{N}(\gamma)}. 
\]
Note that the Bergman-Szeg\"{o} kernel $\Pi_{N}(z,z)$ with $z \in (\mathbb{C}^{*})^{m}$ is written as 
\begin{equation}
\label{BSk}
\Pi_{N}(z,z)=\sum_{\gamma \in S_{N}}
\frac{e^{-N\delta_{S,c}(x)+\ispa{\gamma -Nx,\tau_{S,c}(x)}}}
{Q_{h^{N}}(\gamma)}=\Pi_{N}(x), 
\end{equation}
where $x=\mu (z)$, and hence we have the assertion.  
\end{proof}

Therefore, it would be natural to call the probability measure 
\begin{equation}
\label{BBmeasure}
d\nu_{N}^{x}:=\frac{1}{\Pi_{N}(x)}\sum_{\gamma \in S_{N}}
\frac{m_{S,N}^{\gamma}(x)}{R_{N}(\gamma)}\delta_{\gamma/N}
\end{equation}
the Bergman-Bernstein measure. (The measure $d\nu_{N}^{x}$ defined above 
equals the measure $\mu_{N}^{z}$ with $\mu(z)=x$ in \cite{Z}.) 
Note that the Bergman-Bernstein measure $d\nu_{N}^{x}$ is not a section of 
the barycenter map $b:\mathcal{M}(P) \to P$. In fact, by Lemma \ref{diffBNf} and Theorem \ref{char}, 
it is easy to show the following formula: 
\begin{equation}
\label{BBbary}
\frac{1}{N}D \log \Pi_{N}(x)=b(d\nu_{N}^{x})-x, 
\end{equation}
where, for $f \in C^{\infty}(P)$ and $x \in P$, 
we set $Df(x):=A_{S,c}(x)\nabla f (x)$ with the matrix $A_{S,c}(x)$ defined in \eqref{ASc}, 
and $b(d\nu_{N}^{x})$ denotes the barycenter of the Bergman-Bernstein measure $d\nu_{N}^{x}$. 
Then, it is natural to ask when the Bergman-Bernstein measure $d\nu_{N}^{x}$ coincides 
with the dilated convolution power $d\mathcal{B}_{x}^{N}$ of the 
Bernstein measure $\mathcal{B}_{S,c}(x)$ defined by \eqref{fB}. 
For this question, we have the following proposition. 

\begin{proposition}
\label{BBversusB}
Let $d\nu_{N}^{x}$ be the Bergman-Bernstein measure defined by \eqref{BBmeasure}, 
and let $d\mathcal{B}_{x}^{N}$ denote the dilated convolution power \eqref{convolution} 
of the finitely supported Bernstein measure $\mathcal{B}_{S,c}(x)$ defined in \eqref{fB}. 
Then, the following four conditions are equivalent. 
\begin{enumerate}
\item $R_{N}(\gamma)$ is constant as a function on the finite set $S_{N}$. 
\item The function $\Pi_{N}(x)$ defined in \eqref{kernel} is constant. 
\item The barycenter of $d\nu_{N}^{x}$ is $x$ for each $x \in P$.
\item $d\nu_{N}^{x}=d\mathcal{B}_{x}^{N}$ for each $x \in P$. 
\end{enumerate}
\end{proposition}

\begin{remark}
\label{balanced}
{\rm By \eqref{BSk}, we know that the function $\Pi_{N}(x)$ on $P$ is the function 
induced by the restriction of the Bergman-Szeg\"{o} kernel to the diagonal. 
Hence, the condition (2) in Proposition \ref{BBversusB} is equivalent to 
that the Fubini-Study Hermitian metric $h^{N}$ on $L_{S,c}^{\otimes N}$ is a balanced metric (\cite{D}). }
\end{remark}

\begin{proof}
According to the formula \eqref{BBbary}, it is obvious that the conditions (2) and (3) are equivalent. 
Assume that the condition (1) holds. Since $d\mathcal{B}_{x}^{N}=\sum_{\gamma \in S_{N}}m_{S,N}^{\gamma}(x)\delta_{\gamma/N}$ is 
a probability measure, $\Pi_{N}$ is constant because of its definition \eqref{kernel}, which shows (2).  
Next, assume that the condition (4) holds. Then, we have $\Pi_{N}(x)R_{N}(\gamma)=1$ for each $\gamma \in S_{N}$ and $x \in P$. 
Since $\Pi_{N}$ does not depend on $\gamma$, the condition (1) holds. 
Finally, assume that the condition (2) holds. 
We set $\tilde{m}_{S,N}^{\gamma}(x)=\frac{m_{S,N}^{\gamma}(x)}{\Pi_{N}R_{N}(\gamma)}$ 
so that $d\nu_{N}^{x}=\sum_{\gamma \in S_{N}}\tilde{m}_{S,N}^{\gamma}(x)\delta_{\gamma/N}$. 
Then, by Lemma \ref{diffBNf}, we have 
\begin{equation}
\label{same}
\nabla \tilde{m}_{S,N}^{\gamma}(x) =\tilde{m}_{S,N}^{\gamma}(x) K(x)(\gamma -Nx),  
\end{equation}
where $K(x)=A_{S,c}(x)^{-1}$ is the defining matrix of the Bernstein measure $\mathcal{B}_{S,c}:P \to \mathcal{M}(P)$. 
Let $T_{N}:C(P) \to C(P)$ be the linear map defined by $d\nu_{N}^{x}$, that is, 
\[
T_{N}(f)(x):=\sum_{\gamma \in S_{N}}\tilde{m}_{S,N}^{\gamma}(x)f(\gamma/N),\quad f \in C(P). 
\]
Then, by \eqref{same}, we have 
\[
D_{j}T_{N}(f)=NT_{N}(X_{j}f)-NX_{j}T_{N}(f),\quad f \in C(P), 
\]
where $X_{j}$ is the multiplication operator, $(X_{j}f)(x)=x_{j}f(x)$, and $D_{j}f(x)$ is 
the $j$-th component of $A_{S,c}(x)\nabla f (x)$. 
As in the proof of Proposition \ref{uniqueness}, 
define the first-order differential operators $L_{N,j}$ ($j=1,\ldots,m$) 
by $L_{N,j}=D_{j}+NX_{j}$. Since $[L_{N,j},L_{N,k}]=0$ for each $j,k$, 
we write $L_{N}^{\alpha}=L_{N,1}^{\alpha_{1}}\cdots L_{N,m}^{\alpha_{m}}$ 
for $\alpha=(\alpha_{1},\ldots,\alpha_{m}) \in \mathbb{Z}^{m}_{\geq 0}$. 
Then, as in the proof of Proposition \ref{uniqueness}, we have 
\[
L_{N}^{\alpha}T_{N}(\chi_{\beta})=N^{\|\alpha\|}T_{N}(\chi_{\alpha+\beta}),\ 
\mbox{and hence}\ T_{N}(\chi_{\alpha})=\frac{1}{N^{\|\alpha\|}}L_{N}^{\alpha}\cdot 1,  
\]
where $\chi_{\alpha}$ denotes the monomial with weight $\alpha$. 
Note that the last expression in the above depends only on the matrix $A_{S,c}(x)$, 
and the same formula holds for the Bernstein approximation $B_{N}(\chi_{\alpha})$ instead of $T_{N}(\chi_{\alpha})$.  
Therefore, we have $T_{N}(f)=B_{N}(f)$ for any $f \in C(P)$, which shows the condition (4). 
\end{proof}

One of advantages of using the Bergman-Bernstein approximation $\nu_{N}(f)$ is that it satisfies the identity 
\begin{equation}
\label{ToRiemann}
\int_{P}\Pi_{N}(x)\nu_{N}(f)(x)\,dx=\sum_{\gamma \in NP \cap \mathbb{Z}^{m}}
f(\gamma/N). 
\end{equation}
Then, once one find an asymptotic expansion of $\nu_{N}(f)$ as $N \to \infty$, one would have 
an asymptotic expansion of the Riemann sum \eqref{Riemann} by using a well-known asymptotic 
behavior of the Bergman-Szeg\"{o} kernel function $\Pi_{N}(x)$. This is the idea used in \cite{Z}. 
(More precisely, Zelditch obtains an asymptotic expansion of the numerator in \eqref{BB}.) 
For the Bernstein approximation $B_{N}(f)$ introduced in this paper, we can not, in general, 
expect that the identity like \eqref{ToRiemann} holds for $B_{N}(f)$. 
Instead, our Bernstein measures can be used for general polytopes. 

It might be possible to find asymptotic behavior of the Bergman-Bernstein measure defined by \eqref{BB} 
even for general polytope $P$. However, for this, one might need to 
analyze in detail the behavior of the functions $m_{S,N}^{\gamma}(x)$ 
when $\gamma/N$ and $x$ are close to the boundary of the polytope.

\bibliographystyle{amsalpha}

\end{document}